\documentclass[a4paper,12pt]{article} 
\usepackage{natbib}
\usepackage[utf8]{inputenc}
\usepackage[T1]{fontenc}
\usepackage{lmodern}
\usepackage{amsmath} 
\usepackage{amsthm}
\usepackage{amssymb}
\usepackage{mathtools}
\usepackage{bm}
\usepackage{ifthen}
\usepackage{overpic}
\usepackage{fancyhdr}
\usepackage{graphicx} 
\usepackage[left=1in,top=1in,bottom=1.5in,right=1in]{geometry}

\renewcommand{\subsectionmark}[1]{} 
\newcommand{\mc}[1]{\mathcal{#1}}

\DeclareMathOperator{\re}{Re}			
			
\DeclareMathOperator{\Span}{span}

\newcommand*{\C}{{\mathbb{C}}}     
\newcommand*{\R}{{\mathbb{R}}}     
     
\newcommand*{\N}{{\mathbb{N}}}     
\newcommand*{\Z}{{\mathbb{Z}}}

\newcommand*{\Lin}{{\mathcal{L}}}   
\newcommand*{\Dom}{{\mathcal{D}}}   
\newcommand{\ran}{{\mathcal{R}}}   
\renewcommand{\ker}{{\mathcal{N}}}

\newcommand*{\abs} [1]{\lvert#1\rvert}
\newcommand*{\norm}[1]{\lVert#1\rVert}
\newcommand*{\set} [1]{\{#1\}}
   
\newcommand*{\iprod}[2]{\langle#1,#2\rangle}    
\newcommand*{\dualpair}[2]{\langle#1,#2\rangle} 
\newcommand*{\Set}[1]{\left\{\,#1\,\right\}}
\newcommand*{\Setm}[2]{\left\{\,#1\,\middle|\,#2\,\right\}}
\newcommand*{\Lp}[1][p]{L^{#1}}
\newcommand*{\lp}[1][p]{\ell^{#1}} 

\newcommand{\pmat}[1]{\begin{pmatrix}#1\end{pmatrix}}
\newcommand{\pmatsmall}[1]{\begin{psmallmatrix}#1\end{psmallmatrix}}

\newcommand*{\Abs}[2][default]{\ifthenelse{\equal{#1}{default}}{\left\lvert#2\right\rvert}{\ldelim{#1}{\lvert}#2\rdelim{#1}{\rvert}}}
\newcommand*{\Norm}[2][default]{\ifthenelse{\equal{#1}{default}}{\left\lVert#2\right\rVert}{\ldelim{#1}{\lVert}#2\rdelim{#1}{\rVert}}}

\newcommand*{\Iprod}[3][default]{\ifthenelse{\equal{#1}{default}}{\left\langle#2,#3\right\rangle}{\ldelim{#1}{\langle}#2,#3\rdelim{#1}{\rangle}}}
\newcommand*{\Dualpair}[3][default]{\ifthenelse{\equal{#1}{default}}{\left\langle#2,#3\right\rangle}{\ldelim{#1}{\langle}#2,#3\rdelim{#1}{\rangle}}}

\newcommand*{\List}[2][1]{\set{#1,\ldots,#2}}

\newcommand{\eq}[1]{\begin{align*}#1\end{align*}}
\newcommand{\eqn}[1]{\begin{align}#1\end{align}}

\newcommand{\gs}{\sigma}
\newcommand{\ga}{\alpha}
\newcommand{\gb}{\beta}
\renewcommand{\gg}{\gamma}
\newcommand{\gd}{\delta}
\newcommand{\gl}{\lambda}
\newcommand{\gw}{\omega}

\newcommand{\ieq}[1]{$#1$}

\newcommand{\inv}{^{-1}}

\newcommand*{\ddb}[2][1]{\ifthenelse{\equal{#1}{1}}{\frac{d}{d#2}}{\frac{d^{#1}}{d#2^{#1}}}}
\newcommand*{\pd}[3][1]{\ifthenelse{\equal{#1}{1}}{\frac{\partial{#2}}{\partial{#3}}}{\frac{\partial^{#1}{#2}}{\partial#3^{#1}}}} 
\newcommand*{\limn}[1][n]{\lim_{#1\to\infty}}

\newcommand{\Gconds}{$\mc{G}$-con\-di\-tions}
\newcommand{\yref}{y_{\mbox{\scriptsize\textit{ref}}}}
\newcommand{\blockop}[1]{#1}

\newcommand{\JG}[1][k]{J_{\mc{G}_1}(i\gw_{#1})}
\newcommand{\Ops}{\mc{O}}

\renewcommand{\vec}[1]{\bm{#1}}

\newcommand*{\keyterm}[1]{\emph{#1}}

\newtheorem{theorem}{Theorem}
\newtheorem{lemma}[theorem]{Lemma}
\newtheorem{corollary}[theorem]{Corollary}

\newtheorem{assumption}[theorem]{Assumption}

\theoremstyle{definition}
\newtheorem{remark}[theorem]{Remark} 
\newtheorem{definition}[theorem]{Definition} 

\newenvironment{ORP}{\textbf{The Output Regulation Problem on $W_\ga$.}\it}{}
\newenvironment{RORP}{\textbf{The Robust Output Regulation Problem on $W_\ga$.}\it}{}

\bibpunct{[}{]}{,}{n}{}{,}

\title{The Internal Model Principle for Systems with Unbounded Control and Observation}
\author{Lassi Paunonen\thanks{Tampere University of Technology, PO.Box 553, 33101 Tampere, Finland, \texttt{lassi.paunonen@tut.fi}}
        \and Seppo Pohjolainen\thanks{Department of Mathematics, Tampere University of Technology, PO. Box 553, 33101 Tampere, Finland ({\tt seppo.pohjolainen@tut.fi})}}
\date{~}

\begin{document}

\maketitle
\vspace{-8ex}

\thispagestyle{plain}

\begin{abstract}
In this paper the theory of robust output regulation of distributed parameter systems with infinite-dimensional exosystems is extended for plants with unbounded control and observation. As the main result, we present the internal model principle for linear infinite-dimensional systems with unbounded input and output operators. We do this for two different definitions of an internal model found in the literature, namely, the p-copy internal model and the $\mc{G}$-conditions. We also introduce a new way of defining an internal model for infinite-dimensional systems. The theoretic results are illustrated with an example where we consider robust output tracking for a one-dimensional heat equation with boundary control and pointwise measurements.  
\end{abstract}

\section{Introduction}

The topic of this paper is the theory of robust output regulation for distributed parameter systems. Research in this branch of control of linear systems has been active for over 30 years~\cite{Sch83a,Poh82,HamPoh00,RebWei03,ByrLau00}. The main goal in robust output regulation is to design a control law in such a way that the output $y(t)$ of the system
\begin{subequations}
  \label{eq:plantintro}
  \eqn{
  \dot{x}(t)&= Ax(t)+Bu(t) + w(t), \qquad x(0)=x_0\in X\\
  y(t)& = Cx(t) + Du(t)
  }
\end{subequations}
tracks a given reference signal $\yref(t)$ despite the external disturbance signals $w(t)$. Moreover, the control law needs to be robust with respect to uncertainties in the parameters $(A,B,C,D)$ of the plant. The considered reference and disturbance signals are assumed to be generated by an \keyterm{exosystem} of the form
\begin{subequations}
  \label{eq:exointro}
  \eqn{
  \dot{v}(t)&=Sv(t) \qquad v(0)=v_0\in W\\
  w(t)&=Ev(t)\\
  \yref(t)&=-Fv(t)
  }
\end{subequations}
(the minus sign is for notational convenience). 
With a suitable choice of a finite-dimen\-sional space $W$ and a matrix $S$ with eigenvalues are on the imaginary axis, the class of signals generated by~\eqref{eq:exointro} includes trigonometric functions, polynomials of $t$, and their linear combinations. 
However, if we are interested in nonsmooth reference and disturbance signals, the underlying space $W$ becomes a separable Hilbert space and $S$ is a generator of a strongly continuous group.
In particular, robust tracking and disturbance rejection of any given continuous periodic reference and disturbance signals $\yref^\ast(t)$ and $w^\ast(t)$, respectively, can be formulated as a robust output regulation problem for an infinite-dimensional exosystem~\cite{ImmPoh05b,Pau14a}. Tracking of nonsmooth periodic and almost periodic signals with high accuracy is necessary, e.g., in the control of disk drive systems and robot arms~\cite{Yam93,HarYam88}, and in power electronics~\cite{CosGri09}.

Recent years have seen many succesful efforts in the development of the state space theory of robust output regulation for distributed parameter systems with infinite-dimen\-sional exosystems~\cite{Imm07a,PauPoh10,HamPoh10,PauPoh12a}. In particular, the \keyterm{p-copy internal model principle} of Francis and Wonham~\cite{FraWon75a}, and Davison~\cite{Dav76} was extended for infinite-dimensional linear systems by the current authors in~\cite{PauPoh10}.
This fundamental theorem states that a stabilizing feedback controller solves the robust output regulation problem if and only if it contains a suitable \keyterm{internal model}, i.e., a part that is capable of reproducing the dynamic behavior of the exosystem~\eqref{eq:exointro}.
One of the most important implications of the internal model principle is that the robust output regulation problem can be divided into two parts: One of (i) building an internal model of the exosystem's dynamics into the controller, and (ii) stabilizing the closed-loop system.
This subdivision proves to be especially useful in the case of infinite-dimensional exosystems. For such signal generators exponential closed-loop stability is usually unachievable, and stabilizing the closed-loop system becomes a difficult problem on its own. The internal model principle allows considering the two challenging parts of the main problem separately. In this paper we concentrate on the first subproblem and, in particular, on showing that the internal model in the controller is both necessary and sufficient for the solvability of the robust output regulation problem. 

The purpose of this paper is to extend the theory of robust output regulation and the internal model principle for a larger class of linear systems.
In references~\cite{Imm07a,PauPoh10,HamPoh10,PauPoh12a} the control and observation operators of the plant~\eqref{eq:plantintro} were assumed to be bounded. This standing assumption severely limits the applicability of the theoretic results, because control schemes involving unbounded control and observation are frequently encountered in practical applications. Most notably,
such situations arise
in the control of partial differential equations with boundary control or pointwise measurements~\cite[Ch. 10]{TucWei09}.  We extend the most important parts of the theory presented in the previous references for systems with possibly unbounded $B$ and $C$. 
In the main part of the manuscript we
work under the standing assumption 
that the closed-loop system operator with maximal domain generates a strongly continuous semigroup. This assumption guarantees that the closed-loop system with the dynamic error feedback controller has a well-defined state. 
Subsequently
in Section~\ref{sec:regsys}
we 
show that
the results presented in this paper 
can be used
in the situation where
both the plant and the controller are \keyterm{regular linear systems}~\cite{Wei94,CurWei97,Sta05book}.

In the frequency domain the robust output regulation problem for systems with unbounded control and observation has been considered previously in~\cite{RebWei03,HamPoh00,LogTow97} for finite-dimensional exosystems, and in \cite{WeiHaf99} for a diagonal infinite-dimensional exosystem.
In the state space the robust output regulation for systems with unbounded inputs and outputs has not been considered together with infinite-dimensional exosystems. 
Moreover, the main results of this paper, especially the internal model principle, are also new for an exosystem~\eqref{eq:exointro} on a finite-dimensional space $W=\C^r$. 

Recently in~\cite{PauPoh13a} the robust output regulation problem was studied in a situation where the controller was not required to be robust with respect to \keyterm{all} perturbations to the parameters of the plant, but robustness was instead required with respect to some smaller class of uncertainties. The motivation for this study was that some perturbations of the parameters of the plant may be unrealistic in applications. It was demonstrated in~\cite{PauPoh13a} that there are situations where robustness
(with respect to a smaller class of perturbations) does not 
require a ``full'' internal model in the controller. One of the key results was that the robustness of a controller can be characterized using the solvability of a set of linear equations only involving the transfer function of the plant evaluated at the frequencies of the exosystem, and the operators of the controller. In this paper we extend these results for plants and controllers with unbounded input and output operators. Also, in~\cite{PauPoh13a} the exosystem was assumed to be finite-dimensional and the closed-loop system to be exponentially stabilizable. In this paper we consider an infinite-dimensional exosystem and strongly stabilizable closed-loop systems. Finally, our results also generalize those in~\cite{PauPoh13a} by allowing disturbance signals $w(t)$ to the state of the plant~\eqref{eq:plantintro}. 

The most important contribution of this paper is the extension of the p-copy internal model principle for distributed parameter systems with unbounded input and output operators.
The proof of the internal model principle given in~\cite{PauPoh10} contains parts that can not be extended to the class of systems considered in this paper. Instead, we present a new, more direct proof for the p-copy internal model principle. As a byproduct, the new proof yields a new way of characterizing controllers incorporating an internal model of the exosystem.

We also show that the robustness properties of the controller can equivalently be characterized using the so-called \Gconds\ \cite{HamPoh10,PauPoh10}. The \Gconds\ can be seen as an alternative way of defining an internal model in the controller. The p-copy internal model and the \Gconds\ both have their strengths and weaknesses. In particular, the \Gconds\ 
can be used in characterizing robustness
even if the output space of the plant is infinite-dimensional.

In addition to the unbounded inputs and outputs in the plant~\eqref{eq:plantintro}, we also allow the output operator of the dynamic error feedback controller to be unbounded.
We conjecture that an unbounded operator in the controller will help achieve better stability properties for the closed-loop, especially if the closed-loop system is being stabilized polynomially~\cite{PauPoh13b}.

We conclude the paper by considering robust output regulation for a one-di\-men\-sion\-al heat equation with boundary control and point observation. In the first part of the example, we design a feedback controller with a 2-dimensional internal model to solve the robust output regulation problem for tracking and rejecting constant exogeneous signals. In the second part, we consider tracking of nonsmooth periodic signal using an infinite-dimensional diagonal exosystem. We construct a controller satisfying the \Gconds. The theory presented in this paper shows that the controller solves the robust output regulation problem provided that the remaining parameters of the controller can be chosen in such a way that the closed-loop system is strongly stable.

The organization of the paper is as follows. In Section~\ref{sec:plantexo} we introduce notation, and state the standing assumptions on the plant, the exosystem, and the controller. We also define the class of perturbations considered in robust output regulation. In Sections~\ref{sec:RORP} and~\ref{sec:ORP} we formulate the robust output regulation problem, and show that the solvability of this problem without the requirement of robustness can be characterized using the solvability of regulator equations. Ways of characterizing robustness with respect to a given set of perturbations are studied in Section~\ref{sec:RedIMs}. The p-copy internal model principle is presented in Section~\ref{sec:pcopy}, and in Section~\ref{sec:Gconds} we show that the robustness properties of a controller can also be characterized using the \Gconds. In Section~\ref{sec:regsys} we prove that the results presented in this paper can be used in the situation where the plant and the controller are regular linear systems. In Section~\ref{sec:heatex} we present an example where we design controllers for robust output tracking of a one-dimensional heat equation. Section~\ref{sec:conclusions} contains concluding remarks.

\section{Mathematical Preliminaries}
\label{sec:plantexo}

In this section we introduce the notation and state the assumptions on the plant, the exosystem and the controller. While the input and output operators of the plant and the controller are allowed to be unbounded operators, we assume that the closed-loop system is well-defined in the sense that the closed-loop system operator with maximal domain generates a strongly continuous semigroup.

If $X$ and $Y$ are Banach spaces and $A:X\rightarrow Y$ is a linear operator, we denote by $\Dom(A)$, $\ker(A)$ and $\ran(A)$ the domain, kernel and range of $A$, respectively. The space of bounded linear operators from $X$ to $Y$ is denoted by $\Lin(X,Y)$. If \mbox{$A:X\rightarrow X$,} then $\gs(A)$, $\gs_p(A)$ and $\rho(A)$ denote the spectrum, the point spectrum and the \mbox{resolvent} set of $A$, respectively. For $\gl\in\rho(A)$ the resolvent operator is given by \mbox{$R(\gl,A)=(\gl -A)^{-1}$}. 
The inner product on a Hilbert space and the dual pairing on a Banach space are both denoted by $\iprod{\cdot}{\cdot}$.

For $n\in \N$ we denote $X^n = X\times X\times \cdots \times X$ and $\Dom(A)^n = \Dom(A)\times \cdots \times \Dom(A)$ where a Banach space $X$ and the domain $\Dom(A)$, respectively, are repeated $n$ times. 
If $T\in \Lin(X,Y)$ and $\vec{x} = (x_1, x_2, \ldots, x_n)^T \in X^n$ for some $n\in\N$, then by $\blockop{T}\vec{x}$ we mean that the operator $T$ is applied to all of the components of $\vec{x}$, i.e. $\blockop{T}\vec{x} = (Tx_1, \ldots, Tx_n)^T \in Y^n$.

We consider a linear system~\eqref{eq:plantintro} where  $x(t)\in X$ is the state of the system, $y(t)\in Y$ is the output, and $u(t)\in U$ the input. The spaces $X$, $U$, and $Y$ are Banach spaces.
Here $w(t)\in X$ denotes the disturbance signal to the state of the plant. 
We assume that $A:\Dom(A)\subset X\rightarrow X$ generates a strongly continuous semigroup $T(t)$ on $X$. For a fixed $\gl_0>\gw_0(T(t))$ we define the scale spaces $X_1 = (\Dom(A), \norm{(\gl_0-A)\cdot})$ and $X_{-1}= \overline{(X,\norm{R(\gl_0,A)\cdot})}$ (the completion of $X$ with respect to the norm $\norm{R(\gl_0,A)\cdot}$)~\cite{TucWei09},\cite[Sec. II.5]{EngNag00book}.
We assume the input and output operators of the plant are such that $B\in\Lin(U,X_{-1})$, $C\in\Lin(X_1,Y)$, and the feedthrough operator satisfies $D\in\Lin(U,Y)$. 
We denote by $A_{-1}: X\subset X_{-1}\rightarrow X_{-1}$ and $T_{-1}(t)$ the extensions of the operator $A$ and the semigroup $T(t)$, respectively, to the space $X_{-1}$. 
We assume the operators $B$ and $C$ satisfy $\ran(R(\gl_0,A_{-1})B)\subset \Dom(C)$ 
and $CR(\gl_0,A_{-1})B\in \Lin(U,Y)$
for some/all $\gl_0\in \rho(A)$. 
The transfer function of the system is defined as
\eq{
P(\gl) =CR(\gl,A_{-1})B+D \in \Lin(U,Y)
}
for $\gl\in\rho(A)$.

In the following we construct an infinite-dimensional block diagonal exosystem with frequencies with eigenvalues $(\gw_k)_{k\in \Z}\subset \R$ to generate the reference and disturbance signals. We do this by choosing the parameters of the system~\eqref{eq:exointro} appropriately. The resulting classes of reference and disturbance signals are analyzed in greater detail in~\cite[Sec. 3]{PauPoh12a}.
Let $W$ be a separable Hilbert space with an orthonormal basis
\eq{
\Set{\phi_{k}^l}_{kl}:=\Setm{\phi_k^l\in W}{k\in \Z,~ l=1,\dots,n_k}.
}
More precisely, we have
$W=\overline{\Span}\Set{\phi_k^l}_{kl}$ and $\iprod{\phi_k^l}{\phi_n^m}=\gd_{kn}\gd_{lm}$.
The lengths~$n_k\in\N$ of the subsequences are uniformly bounded. 
For given $(\gw_k)_{k\in\Z}\subset \R$ the operators $S_k\in \Lin(W)$ representing finite-dimen\-sion\-al Jordan blocks 
are defined as
\eq{
S_k=i\gw_k\iprod{\cdot}{\phi_k^1}\phi_k^1 +\sum_{l=2}^{n_k}\iprod{\cdot}{\phi_k^l} \left(i\gw_k\phi_k^l +\phi_k^{l-1}\right).
}
The operators $S_k$ have the property that
$(i\gw_k-S_k)\phi_k^1=0$, and
$(S_k-i\gw_k)\phi_k^{l}=\phi_k^{l-1}$ for all $l\in\List[2]{n_k}$.
The system operator $S$ is defined as
\eq{
Sv=\sum_{k\in\Z}S_kv, \qquad \Dom(S)=\biggl\{ v\in W \biggm| \sum_{k\in\Z}\,\norm{S_kv}^2<\infty \biggr\}.
}

The spectrum of the operator~$S$ satisfies
\ieq{
\gs(S)=\overline{\gs_p(S)}=\overline{\Set{i\gw_k}}_{k\in \Z}.
}
The operator $S$ generates a strongly continuous group $T_S(t)$ on $W$, and
\eq{
T_S(t)v
=\sum_{k\in\Z}e^{i\gw_kt}\sum_{l=1}^{n_k} \iprod{v}{\phi_k^l}\sum_{j=1}^l \frac{t^{l-j}}{(l-j)!}\phi_k^j ,
}
for all $v\in W$, and $t\in\R$.
For any $n_S\in \N$ such that $n_S\geq n_k$ for all $k\in\Z$
there exists~$M_S\geq 1$ such that  
\ieq{
\norm{T_S(t)}\leq M_S(\abs{t}^{n_S}+1)
}
for all $t\in\R$.
The operators $E$ and $F$ are assumed to be bounded in such a way that $E\in \Lin(W,X)$ and $F\in \Lin(W,Y)$.

For $k\in\Z$ we define the orthogonal projection
\ieq{
P_k=\sum_{l=1}^{n_k}\iprod{\cdot}{\phi_k^l}\phi_k^l
} 
onto the finite-dimensional subspace $\Span\set{\phi_k^l}_{l=1}^{n_k}$ of~$W$. With this notation the domain of the operator $S$ satisfies
\eq{
\Dom(S)
&= \Setm{v\in W}{\sum_{k\in\Z} (1+\gw_k^2) \norm{P_kv}^2<\infty} .
} 
We define scale spaces $W_\ga\subset W$ related to the system operator $S$ of the exosystem. 

\begin{definition}
    \label{def:Wga}
    For $\ga\geq 0$ we denote by $(W_\ga,\norm{\cdot}_\ga)$ the Hilbert space
  \eq{
  W_\ga=\Setm{v\in W}{\sum_{k\in\Z}(1+\gw_k^2)^\ga \norm{P_kv}^2<\infty}
  }
  with norm $\norm{\cdot}_\ga$ defined by
  \ieq{
  \norm{v}_\ga^2 =\sum_{k\in\Z}(1+\gw_k^2)^\ga \norm{P_kv}^2 
  }
  for $v\in W_\ga$.
\end{definition}

The spaces $W_\ga$ are invariant under the group~$T_S(t)$, the restrictions $T_S(t)\vert_{W_\ga}$ are strongly continuous groups on $W_\ga$ and the generators of these groups are $S\vert_{W_\ga}:\Dom(S\vert_{W_\ga})\subset W_\ga\rightarrow W_\ga$ with domains $\Dom(S\vert_{W_\ga})=W_{\ga+1}$.

\begin{remark}
  \label{rem:findimexo}
  The results in this paper are presented for infinite-dimensional block diagonal exosystems. However, the main results are also new for systems with unbounded $B$ and $C$ together with a finite-dimensional exosystem of the form~\eqref{eq:exointro} on a finite-dimensional space $W=\C^r$. In this situation the operator $S$ is a matrix in its Jordan canonical form with distinct eigenvalues $\gs(S)=\set{i\gw_k}_{k=1}^q$. The orthonormal basis $\set{\phi_k^l}_{kl}$ of $W$ can be chosen to consists of Euclidean basis vectors $\set{e_k}_{k=1}^r\subset \C^r$ 
in such a way that
  \eq{
  (e_1,\ldots,e_r) 
  = (\phi_1^1,\ldots,\phi_1^{n_1},\phi_2^1,\ldots,\phi_2^{n_2},\ldots,\phi_q^1,\ldots\phi_q^{n_q}),
  }
  where $n_k\in \N$ is the size of the Jordan block $S_k$ associated to the 
  eigenvalue $i\gw_k$ in $S$.
If the exosystem is finite-dimensional, then
many of
the proofs in this paper become simpler due to the fact that the infinite index set $k\in\Z$ is replaced by the finite set $k\in \List{q}$ of indices.
  For a finite-dimensional exosystem we also have $W_\ga = W$ for every $\ga\geq 0$.
  \end{remark}

We consider a dynamic error feedback controller of the form
\eq{
\dot{z}(t) &= \mc{G}_1z(t)+\mc{G}_2e(t), \qquad z(0)=z_0\in Z,\\
u(t) &= Kz(t)
}
on a Banach space $Z$. The operator $\mc{G}_1:\Dom(\mc{G}_1)\subset Z\rightarrow Z$ generates a strongly continuous semigroup $T_{\mc{G}_1}(t)$ on $Z$, 
and the scale space $Z_1$ is defined similarly as for the plant. We assume
$\mc{G}_2\in\Lin(Y,Z)$
and $K\in\Lin(Z_1,U)$. 

The system and the controller can be written together as a closed-loop system on the Banach space $X_{e}=X\times Z$. This composite system with state $x_e(t)=(x(t), z(t))^T$ can be written formally on $X_{-1}\times Z$ as
  \eq{
  \dot{x}_e(t) &= A_ex_e(t)+B_ev(t), \qquad x_e(0)=x_{e0} ,\\
  e(t) &= C_ex_e(t)+D_ev(t),
  }
where $e(t)= y(t)- \yref(t)$ is the \keyterm{regulation error}, $x_{e0}=(x_0, z_0)^T$,  $C_e= \bigl(C, ~ DK\bigr)$, 
$D_e=F$,
  \eq{
  A_e=\pmat{A_{-1}&BK\\\mc{G}_2C&\mc{G}_1+\mc{G}_2DK}, \qquad 
  B_e=\pmat{E\\\mc{G}_2F}.
  }
Due to the unboundedness of the operators $B$, $C$, and $K$ the domain of the operators $A_e$ will not be $\Dom(A)\times \Dom(\mc{G}_1)$ as in references~\cite{HamPoh10,PauPoh12a}. 
Instead, we consider the maximal domain such that $A_e$ is an operator on $X_e$, i.e., maximal domain for which $\ran(A_e)\subset X_e=X\times Z$.
Since $\mc{G}_2Cx + (\mc{G}_1+\mc{G}_2DK)z\in Z$ if and only if $x\in \Dom(C)$ and $z\in \Dom(\mc{G}_1)$, this domain is given by
\eq{
\Dom(A_e) = \biggl\{ \pmat{x\\z}\in \Dom(C)\times \Dom(\mc{G}_1)~\biggm|
~A_{-1}x+BKz\in X\biggr\}.
}
The operator $C_e$ is unbounded with domain $\Dom(C_e) = \Dom(C)\times \Dom(K)\supset \Dom(A_e)$ and $B_e \in \Lin(W,X\times Z)$

\begin{assumption}
  \label{ass:CLsysass}
  Throughout the paper we assume $(A,B,C,D)$ and $(\mc{G}_1,\mc{G}_2,K)$ are such that $A_e$ with the given domain generates a strongly continuous semigroup $T_e(t)$ on $X_e$, and that $C_e$ is relatively bounded with respect to $A_e$.
\end{assumption}

Later in Section~\ref{sec:regsys} we show that Assumption~\ref{ass:CLsysass} is in particular satisfied if the plant and the controller are regular linear systems.
If $\gl_0\in\rho(A_e)$, then the $A_e$-boundedness of $C_e$ is equivalent to the condition $C_e (\gl_0 -A_e)\inv\in \Lin(X_e,Y)$.

\subsection{The Class of Perturbations}
\label{sec:pertclass}

In this paper we consider a situation where parameters of the plant are perturbed in such a way that the operators $A$, $B$, $C$, and $D$ are changed into $\tilde{A}: \Dom(\tilde{A})\subset X\rightarrow X$, $\tilde{B}\in \Lin(U,\tilde{X}_{-1})$, $\tilde{C}\in \Lin(\tilde{X}_1,Y)$, and $\tilde{D}\in \Lin(U,Y)$, respectively. Here $\tilde{X}_1$ and $\tilde{X}_{-1}$ are the scale spaces of $X$ related to the operator $\tilde{A}$.
Moreover, the operators $E$ and $F$ are perturbed in such a way that $\tilde{E}\in \Lin(W,X)$ and $\tilde{F}\in \Lin(W,Y)$.
For $\gl\in \rho(\tilde{A})$ we denote by $\tilde{P}(\gl) = \tilde{C} R(\gl, \tilde{A}_{-1}) \tilde{B} + \tilde{D}$ the transfer function of the perturbed plant.
We likewise denote the operators of the closed-loop system consisting of the perturbed plant and the controller by $\tilde{C}_e=\bigl(\tilde{C},~\tilde{D}K\bigr)$, $\tilde{D}_e = \tilde{F}$ and
\eq{
\tilde{A}_e=\pmat{\tilde{A}_{-1}&\tilde{B}K\\\mc{G}_2\tilde{C}&\mc{G}_1+\mc{G}_2\tilde{D}K},
\qquad \tilde{B}_e = \pmat{\tilde{E}\\\mc{G}_2 \tilde{F}}.
}

\begin{assumption}
  \label{ass:pertclass}
  The perturbations $(\tilde{A},\tilde{B},\tilde{C},\tilde{D},\tilde{E},\tilde{F})$ in the class $\Ops$ of considered perturbations are assumed to satisfy the following conditions:
  \begin{itemize}
    \item[\textup{(a)}] The perturbed system operator $\tilde{A}$ generates a strongly continuous semigroup on $X$ and satisfies $i\gw_k \in \rho(\tilde{A})$ for all $k\in \Z$. 
      The operators $\tilde{B}$ and $\tilde{C}$ are such that $\ran(R(\gl_0,\tilde{A}_{-1})\tilde{B})\subset \Dom(\tilde{C})$ 
      and $\tilde{C}R(\gl,_0,\tilde{A}_{-1})\tilde{B}\in \Lin(U,Y)$
      for some/all $\gl_0\in \rho(\tilde{A})$.  
    \item[\textup{(b)}] The perturbed closed-loop system operator $\tilde{A}_e$ with maximal domain generates a strongly stable strongly continuous semigroup on $X_e$ and $\tilde{C}_e$ is $\tilde{A}_e$-bounded.
    \item[\textup{(c)}] The Sylvester equation $\Sigma S = \tilde{A}_e \Sigma + \tilde{B}_e$ has a solution $\Sigma \in \Lin(W_\ga,X_e)$ satisfying $\Sigma(W_{\ga+1})\subset \Dom(\tilde{A}_e)$.
  \end{itemize}
\end{assumption}

If the unperturbed closed-loop system is exponentially stable, then the conditions of Assumption~\ref{ass:pertclass} are satisfied, in particular, for any bounded perturbations of small enough norms.
If the exosystem is finite-dimensional (see Remark~\ref{rem:findimexo}), then the Sylvester equation $\Sigma S = A_e\Sigma + B_e$ has a solution $\Sigma \in \Lin(W,X_e)$ satisfying $\ran(\Sigma)\subset \Dom(A_e) $ provided that $\gs(A_e)\cap \gs(S)= \varnothing $~\cite{Vu91}. Likewise, part (c) of Assumption~\ref{ass:pertclass} is satisfied whenever $\gs(\tilde{A}_e)\cap \gs(S)=\varnothing$.

\subsection{Special Operators}
\label{sec:SpecialOps}

To state some of the main results of the paper, we need additional notation. 
For $k\in \Z$ and $n\in\N$ we define the operator 
$\JG: \Dom(\mc{G}_1)^n \subset Z^n \rightarrow Z^n$
to be a block upper triangular operator with diagonal elements $i\gw_k - \mc{G}_1$ and identity operators $I$ on the first superdiagonal, i.e.,
\eq{
\JG = \pmatsmall{i\gw_k - \mc{G}_1&I\\&i\gw_k - \mc{G}_1\\&&\ddots &I\\ &&&i\gw_k - \mc{G}_1}.
}
The form of the operator $\JG$ immediately implies that for all $\vec{z} = (z_{n_k},\ldots,z_1)^T \in \Dom(\mc{G}_1)^{n_k}$ such that $\vec{z}\neq 0$ the condition $\JG \vec{z} = 0$ is equivalent to $(z_l)_{l=1}^{n_k}$ forming a Jordan chain of $\mc{G}_1$ associated to the eigenvalue $i\gw_k$, i.e. $(i\gw_k - \mc{G}_1)z_1 = 0$ and $(\mc{G}_1 - i\gw_k)z_{l} = z_{l-1}$ for $l\in \List[2]{n_k}$.

For $k\in\Z$ and for an operator $\tilde{A}$ denote $\tilde{R}_k = R(i\gw_k,\tilde{A}_{-1})$. 
We define a block triangular operator $\mathbb{R}(i\gw_k, \tilde{A}_{-1})\in \Lin(X^{n_k})$ by
\eq{
\mathbb{R}(i\gw_k,\tilde{A}_{-1}) 
= \pmatsmall{\tilde{R}_k & -\tilde{R}_k^2& \cdots & (-1)^{n_k-1} \tilde{R}_k^{n_k} \\
&\tilde{R}_k &  \cdots & (-1)^{n_k-2} \tilde{R}_k^{n_k-1} \\
&&\ddots&\vdots\\
&&&\tilde{R}_k
}.
}
For $k\in\Z$ and for operators $\tilde{A}$, $\tilde{B}$, $\tilde{C}$, and $\tilde{D}$ satisfying $i\gw_k\in \rho(\tilde{A})$, we denote by $\tilde{\mathbb{P}}(i\gw_k) \in \Lin(U^{n_k},Y^{n_k})$ the operator 
\eq{
\tilde{\mathbb{P}}(i\gw_k) 
= \pmatsmall{\tilde{P}(i\gw_k) & -\tilde{C}\tilde{R}_k^2\tilde{B}& \cdots & (-1)^{n_k-1} \tilde{C}\tilde{R}_k^{n_k} \tilde{B}\\
&\tilde{P}(i\gw_k) &  \cdots & (-1)^{n_k-2} \tilde{C}\tilde{R}_k^{n_k-1} \tilde{B}\\
&&\ddots & \vdots\\[.7ex]
&&&\tilde{P}(i\gw_k)
}
= \blockop{\tilde{C}}
\mathbb{R}(i\gw_k,\tilde{A}_{-1} )
\blockop{\tilde{B}} + \blockop{\tilde{D}}.
}
For the operators $A$, $B$, $C$, and $D$ of the nominal plant, we use the notation $\mathbb{P}(i\gw_k)$.
Finally, for $k\in \Z$ we define $\Phi_k = (\phi_k^{n_k}, \phi_k^{n_k-1}, \ldots, \phi_k^1)^T \in W^{n_k}$.

It should be noted that if for some $k\in \Z$ we have $n_k=1$, then the above operators reduce to $\JG = i\gw_k - \mc{G}_1$, $\mathbb{R}(i\gw_k,\tilde{A}_{-1})= R(i\gw_k,\tilde{A}_{-1})$, and $\tilde{\mathbb{P}}(i\gw_k)=\tilde{P}(i\gw_k)$.

\section{Control Objectives}
\label{sec:RORP}

In this section we formulate the robust output regulation problem. 
The problem statement depends on the parameter $\ga>0$. In particular, the decay of the regulation error is required only for the reference and disturbance signals corresponding to the initial states $v_0\in W_{\ga+1}$ of the exosystem.
As shown in~\cite[Sec. 3]{PauPoh12a}, in the case of the periodic reference and disturbance signals the choices of the initial states of the exosystem are directly related to the level of smoothness of the signals to be tracked and rejected.

\begin{RORP}
  Choose the controller $(\mc{G}_1,\mc{G}_2,K)$ in such a way that the following are satisfied:
\begin{itemize}
  \item[\textup{(a)}] The closed-loop system operator $A_e$ generates a strongly stable semigroup.
  \item[\textup{(b)}] 
  For all initial states $x_{e0}\in \Dom(A_e)$ and  $v_0\in W_{\ga+1}$ the regulation error decays to zero asymptotically, i.e., $e(t)\rightarrow 0$ as $t\rightarrow \infty$.
\item[\textup{(c)}] If the operators $(A,B,C,D,E,F)$ are perturbed to $(\tilde{A},\tilde{B},\tilde{C},\tilde{D},\tilde{E},\tilde{F})\in \Ops$ (i.e. the perturbed closed-loop system is strongly stable and additional assumptions made in Section~\textup{\ref{sec:pertclass}} are satisfied),
   then for all initial states $x_{e0}\in \Dom(\tilde{A}_e)$ and  $v_0\in W_{\ga+1}$ the regulation error satisfies $e(t)\rightarrow 0$ as $t\rightarrow \infty$.
\end{itemize}
\end{RORP}

The parts (a) and (b)  of the robust output regulation problem (i.e., the problem without the requirement for robustness) are referred to as the \keyterm{output regulation problem}.

\begin{ORP}
  Choose the controller $(\mc{G}_1,\mc{G}_2,K)$ in such a way that parts \textup{(a)} and \textup{(b)} of the robust output regulation problem are satisfied.
\end{ORP}

\section{Characterizing the Solvability of the Output Regulation Problem}
\label{sec:ORP}

In this section we show that the solvability of the output regulation problem can be characterized using the solvability of the so-called \keyterm{regulator equations}~\cite{FraWon75a,ByrLau00}.

\begin{theorem}
  \label{thm:ORP} 
  Assume the controller\/ $(\mc{G}_1,\mc{G}_2,K)$ is such that
  $A_e$ generates a strongly stable semigroup on~$X_e$, and that the Sylvester equation 
$\Sigma S=A_e\Sigma +B_e  $
 on $W_{\ga+1}$ has a solution $\Sigma\in \Lin(W_\ga,X_e)$.
  Then the following are equivalent:
\begin{enumerate}
\item[\textup{(a)}] The controller\/ $(\mc{G}_1,\mc{G}_2,K)$ solves the output regulation problem on~$W_\ga$.
\item[\textup{(b)}] The regulator equations
  \begin{subequations}
    \label{eq:regeqns}
    \eqn{
    \label{eq:ORPSyl}
    \Sigma S&=A_e\Sigma +B_e  \\
    0&=C_e\Sigma +D_e \label{eq:ORPRegconstr}
    }
  \end{subequations}
  on $W_{\ga+1}$ have a solution $\Sigma\in \Lin(W_\ga,X_e)$.
\end{enumerate}
\end{theorem}

For the proof of the theorem we need some auxiliary results. In particular, Lemma~\ref{lem:ORPCLstateform} shows that the state of the closed-loop system and the regulation error can be expressed using the solution $\Sigma$ of the Sylvester equation~\eqref{eq:ORPSyl}.

\begin{lemma}
  \label{lem:CSigmarelbdd}
  If $1\in \rho(A_e)$ and if $\Sigma\in \Lin(W_\ga,X_e)$ is the solution of~\eqref{eq:ORPSyl}, then $C_e\Sigma \in \Lin(W_{\ga+1},Y)$.
\end{lemma}

\begin{proof}
  Let $v\in W_{\ga+1}$. The Sylvester equation~\eqref{eq:ORPSyl} implies $\Sigma (S-I)v = (A_e- I)\Sigma v+ B_e v$. Now $\Sigma v \in \Dom(A_e)\subset \Dom(C_e)$ and using~\eqref{eq:ORPSyl} we have
  \eq{
  \MoveEqLeft[2.7]\norm{C_e\Sigma v} 
  = \norm{C_e (A_e - I)\inv (A_e -  I) \Sigma v}
  =  \norm{C_e (A_e -  I)\inv (\Sigma (S-I) v - B_e v)}\\
&\leq  \norm{C_e (A_e-I)\inv}_{\Lin(X_e,Y)} \left( \norm{\Sigma}_{\Lin(W_\ga,X_e)} \norm{(S-I)v}_\ga 
+ \norm{B_e}_{\Lin(W_\ga,X_e)} \norm{v}_\ga \right),
  }
  which implies that $C_e\Sigma \in \Lin(W_{\ga+1},Y)$.
\end{proof}

\begin{lemma}
  \label{lem:ORPCLstateform}
Let $\Sigma \in \Lin(W_\ga,X_e)$ be a solution of the Sylvester equation~\eqref{eq:ORPSyl}.
    For all initial states $x_{e0}\in X_e$ and $v_0\in W$ and for all $t\geq 0$  the state of the closed-loop system satisfies
      \ieq{
      x_e(t)= T_e(t)(x_{e0}-\Sigma v_0)+\Sigma v(t),
      }
and for all $x_{e0}\in \Dom(A_e)$ and $v_0\in W_{\ga+1}$ the regulation error is given by
      \eq{
      e(t)&=C_eT_{e}(t)(x_{e0}-\Sigma v_0) +(C_e\Sigma +D_e)v(t).
      }
  If $x_{e0}\in \Dom(A_e)$ and $v_0\in W_{\ga+1}$, then the regulation error $e(t)$ is continuous and satisfies
  \ieq{
  \norm{e(t) - (C_e\Sigma + D_e)T_S(t)v_0} \rightarrow 0
  }
  as $t\rightarrow \infty$.
\end{lemma}

\begin{proof}
  Let $v\in W_{\ga+1}$. Then  $\Sigma v\in \Dom(A_e)$ and 
  for all $t>s$ we have
  \eq{
  \MoveEqLeft T_{e}(t-s)B_eT_S(s)v= T_{e}(t-s)(\Sigma S-A_e\Sigma ) T_S(s)v \\[1ex]
  &= -T_{e}(t-s)A_e\Sigma  T_S(s)v +T_{e}(t-s)\Sigma ST_S(s)v
  = \ddb{s}\left(T_{e}(t-s)\Sigma  T_S(s)v\right).
  }
  Integrating both sides of this equation from $0$ to $t>0$ gives
  \eqn{
  \label{eq:Syleqnmildform}
  \int_0^t T_{e}(t-s)B_eT_S(s)vds=\Sigma T_S(t)v-T_{e}(t)\Sigma v.
  }
  Since the operators on both sides of this equation are in $\Lin(W_\ga,X_{e})$ and since $W_{\ga+1}$ is dense in $W_\ga$, we have that~\eqref{eq:Syleqnmildform} holds for all $v\in W_\ga$ and $t>0$.

For all $x_{e0}\in X_e$ and $v_0\in W_\ga$ the mild state of the closed-loop system is given by
  \eq{
  x_e(t)=T_e(t)x_{e0} + \int_0^t T_e(t-s)B_eT_S(s)v_0ds.
  }
  We can now use~\eqref{eq:Syleqnmildform} to conclude that
\eq{
  x_e(t) &=T_e(t)x_{e0} + \Sigma T_S(t)v_0-T_{e}(t)\Sigma v_0
   = T_e(t)(x_{e0}-\Sigma v_0) + \Sigma T_S(t)v_0.
  }
  If $x_{e0}\in \Dom(A_e)$ and $v_0\in W_{\ga+1}$, then $\Sigma T_S(t)v_0\in \Dom(A_e)\subset \Dom(C_e)$ for all $t\geq 0$ and
the regulation error is given by 
\eq{
e(t) &=C_ex_e(t)+D_ev(t) 
=C_eT_{e}(t)(x_{e0}-\Sigma v_0) +(C_e\Sigma +D_e)T_S(t)v_0.
}
Since $t\mapsto T_S(t)v_0\in W_{\ga+1}$ is continuous and since by Lemma~\ref{lem:CSigmarelbdd} we have $C_e\Sigma + D_e\in \Lin(W_{\ga+1},Y)$, we can see that $t\mapsto (C_e\Sigma + D_e)T_S(t)v_0$ is continuous. Since we have
\eq{
 C_eT_{e}(t)(x_{e0}-\Sigma v_0) 
&= C_e(A_e-I)\inv (A_e-I)T_e(t)(x_{e0}-\Sigma v_0)\\
&= C_e(A_e-I)\inv T_e(t)(A_e-I)(x_{e0}-\Sigma v_0)
}
where $C_e(A_e-I)\inv\in \Lin(X_e,Y)$, we can conclude that $e(t)$ is continuous. Moreover, 
\eq{
\MoveEqLeft[2]\norm{e(t)-(C_e\Sigma +D_e)T_S(t)v_0}
= \norm{ C_eT_{e}(t)(x_{e0}-\Sigma v_0)}\\
&\leq \norm{ C_e(A_e-I)\inv } \norm{T_{e}(t)(A_e-I)(x_{e0}-\Sigma v_0)}
\rightarrow 0
}
as $t\rightarrow \infty$ due to the strong stability of $T_e(t)$.
\end{proof}

We can now use the previous results to prove Theorem~\ref{thm:ORP}. 

\textit{Proof of Theorem\/ \textup{\ref{thm:ORP}}.}
We will first show that (b)~implies~(a). Assume the regulator equations~\eqref{eq:regeqns} have a solution $\Sigma\in \Lin(W_\ga,X_e)$. 
Since $T_{e}(t)$ is strongly stable, we have from Lemma~\ref{lem:ORPCLstateform} that for all initial states $x_{e0}\in \Dom(A_e)$ and $v_0\in W_{\ga+1}$
\eq{
\limn[t]\,\Norm{e(t)}=\limn[t]\,\Norm{e(t)-(C_e\Sigma +D_e)v(t)}=0,
}
since $C_e\Sigma +D_e=0$ on $W_{\ga+1}$. Thus the controller solves the output regulation problem on $W_\ga$.

It remains to prove that (a)~implies~(b). Assume the controller solves the output regulation problem on $W_\ga$ and $\Sigma \in\Lin(W_\ga,X_{e})$ is a solution of the Sylvester equation~\eqref{eq:ORPSyl} on~$W_{\ga+1}$. Since the regulation error decays to zero asymptotically for all initial states of the closed-loop system and the exosystem, 
Lemma~\ref{lem:ORPCLstateform} implies that for all $x_{e0}\in \Dom(A_e)$ and $v_0\in W_{\ga+1}$ we must have
\eq{
 \norm{(C_e\Sigma +D_e)T_S(t)v_0}
\leq \norm{(C_e\Sigma +D_e)T_S(t)v_0-e(t)}+\norm{e(t)}\stackrel{t\rightarrow \infty}{\longrightarrow} 0,
}
and thus $\limn[t](C_e\Sigma +D_e)T_S(t)v_0=0$ for every $v_0\in W_{\ga+1}$. 
Since $C_e\Sigma + D_e \in \Lin(W_{\ga+1},Y)$, we have from
Lemma~\ref{lem:exonondecay} that $\Sigma$ satisfies equation~\eqref{eq:ORPRegconstr}.
\hfill\endproof

\subsection{Properties of the Sylvester Equation $\Sigma S = A_e\Sigma + B_e$}

We conclude the section by stating some relevant properties of the Sylvester equation in Theorem~\ref{thm:ORP}. It should be noted that there are more convenient sufficient conditions for the solvability of the equation than the one given in Theorem~\ref{thm:Sylprop}(b). In particular, this is the case if the norms $\norm{R(i\gw,A_e)}$ are polynomially bounded with respect to $\abs{\gw}$ for $\gw\in\R$~\cite{PauPoh12a,PauPoh13b}. If $X$ and $Z$ are Hilbert spaces, this is equivalent to the closed-loop system being polynomially stable~\cite{BorTom10}.
Also, if the exosystem is finite-dimensional, then $S$ is a bounded operator and the Sylvester equation $\Sigma S = A_e \Sigma + B_e$ has a unique bounded solution
whenever $\gs(A_e)\cap \gs(S)=\varnothing$~\cite{Vu91}.

\begin{theorem}
  \label{thm:Sylprop}
Assume the closed-loop system is strongly stable and let $\ga\geq 0$. Then the Sylvester equation
\eqn{
\label{eq:Sylprops}
\Sigma S = A_e\Sigma + B_e
}
has the following properties.
\begin{itemize}
  \item[\textup{(a)}] The equation~\eqref{eq:Sylprops} may have at most one solution.
  \item[\textup{(b)}] If $i\gw_k\in \rho(A_e)$ for all $k\in\Z$, and if 
    \eqn{
    \label{eq:seriesass}
    \hspace{-1ex}{\sup_{\norm{x_e'}\leq 1}}~{\sum_{k\in\Z}\frac{1}{(1+\gw_k^2)^\ga}\sum_{l=1}^{n_k} \Abs{\sum_{j=1}^l (-1)^{l-j}\dualpair{R(i\gw_k,A_e)^{l+1-j}B_e\phi_k^j}{x_e'} }^2<\infty,}
    \hspace{-3ex}
    }
    where $x_e'\in X_e'$, the dual space of~$X_e$, then~\eqref{eq:Sylprops} has a unique solution $\Sigma\in \Lin(W_\ga,X_e)$ satisfying $\Sigma (W_{\ga+1})\subset \Dom(A_e)$. The solution is given by
    \eqn{
    \label{eq:SylpropsSol}
\Sigma v=\sum_{k\in\Z}\sum_{l=1}^{n_k} \iprod{v}{\phi_k^l}\sum_{j=1}^l(-1)^{l-j} R(i\gw_k,A_e)^{l+1-j}B_e\phi_k^j, \qquad v\in W_\ga.
    } 
  \item[\textup{(c)}] If $i\gw_k\in\rho(A_e)$ for all $k\in\Z$ and if~\eqref{eq:Sylprops} has a solution $\Sigma\in \Lin(W_\ga,X_e)$, then the condition~\eqref{eq:seriesass} is satisfied, and $\Sigma$ is given by~\eqref{eq:SylpropsSol}.
  \item[\textup{(d)}] If~\eqref{eq:Sylprops} has a solution $\Sigma\in \Lin(W_\ga,X_e)$, then for every $k\in\Z$ the equation $\Sigma_k S = A_e \Sigma_k + B_e P_k$ has a unique solution $\Sigma_k = \Sigma P_k\in \Lin(W,X_e)$.
\end{itemize}
\end{theorem}

\begin{proof}
For the proof of part (a) 
  let $\Sigma_1,\Sigma_2\in \Lin(W_\ga,X_e)$ be two solutions of the Sylvester equation. 
  We have
  \eq{
 \left\{
 \begin{array}{l}
  \Sigma_1 S = A_e\Sigma_1  +B_e \\
  \Sigma_2 S = A_e\Sigma_2  +B_e 
   \end{array}
 \right. 
 \qquad\Rightarrow \qquad 
  (\Sigma_1-\Sigma_2) S = A_e(\Sigma_1-\Sigma_2)
  }
  on $W_{\ga+1}$.
  Denote $\Delta = \Sigma_1-\Sigma_2$. 
  For all $t>0$ and $v\in W_{\ga+1}$
  \eq{
 \Delta T_S(t) v - T_e(t) \Delta v 
  &= \Bigl[ T_e(t-s)\Delta T_S(s) v \Bigr]_{s=0}^t
  = \int_0^t \ddb{s} \Bigl( T_e(t-s)\Delta T_S(s) v \Bigr)ds\\
  &= \int_0^t  T_e(t-s) \left( -A_e \Delta + \Delta S \right) T_S(s) v ds
  = 0
  }
  and thus
  \ieq{
  (\Sigma_1-\Sigma_2)T_S(t)v = T_e(t) (\Sigma_1-\Sigma_2)  v
  }
for all $t\geq 0$. Since for all $t\geq 0$ the operators on both sides of the equation are in $\Lin(W_{\ga},X_e)$ and since $W_{\ga+1}$ is dense in $W_\ga$, the above identity holds for all $v\in W_\ga$.

Since $T_e(t)$ is strongly stable, for all $v\in W_\ga$ we have 
  \eq{
 \norm{(\Sigma_1-\Sigma_2)T_S(t)v_0}
= \norm{ T_e(t) (\Sigma_1-\Sigma_2)  v_0 }
\rightarrow 0 
  }
  as $t\rightarrow \infty$. Since $\Sigma_1-\Sigma_2\in \Lin(W_\ga,X_e)$, Lemma~\ref{lem:exonondecay}  implies that $\Sigma_1-\Sigma_2=0$.
  This concludes that the Sylvester equation may have at most one solution. 

  We will next prove part (b). Since $i\gw_k\in\rho(A_e)$ for all $k\in\Z$, we have that $B_e\phi_k^l\in\ran(i\gw_k -A_e)^{n_k-l+1}$ for every $k\in\Z$ and  $l\in\List{n_k}$. Since~\eqref{eq:seriesass} is satisfied, we have from Lemma~3.2 in~\cite{PauPoh10} that the Sylvester equation~\eqref{eq:Sylprops} has a solution $\Sigma\in \Lin(W_\ga,X_e)$ given by~\eqref{eq:SylpropsSol} (in~\cite{PauPoh10} $\ga$ was assumed to be an integer, but the result remains valid for all nonnegative $\ga$).

In order to prove (c)  assume that $i\gw_k\in\rho(A_e)$ for all $k\in\Z$ and that the Sylvester equation has a solution $\Sigma \in \Lin(W_\ga,X_e)$. Let $k\in\Z$.
  Applying both sides of the equation $\Sigma S = A_e\Sigma + B_e$ to the elements $\set{\phi_k^l}_{l=1}^{n_k}$, we obtain
    \eq{
    (i\gw_k - A_e) \Sigma \phi_k^1 &= B_e\phi_k^1\\
    (i\gw_k - A_e) \Sigma \phi_k^2 + \Sigma \phi_k^1 &= B_e\phi_k^2  \\
    \nonumber
    &\vdots\\
    (i\gw_k - A_e) \Sigma \phi_k^{n_k} + \Sigma \phi_k^{n_k-1} &= B_e\phi_k^{n_k} 
    }
  Solving the equations recursively shows that for any $v\in P_kW = \Span\set{\phi_k^l}_{l=1}^{n_k}$ we have
  \eq{
\Sigma v=
\sum_{l=1}^{n_k} \iprod{v}{\phi_k^l}\sum_{j=1}^l(-1)^{l-j} R(i\gw_k,A_e)^{l+1-j}B_e\phi_k^j.
  }
  Since $k\in\Z$ was arbitrary, we have that the operator defined by~\eqref{eq:SylpropsSol} is equal to the unique solution of the Sylvester equation~\eqref{eq:Sylprops} on all subspaces $P_kW$. Therefore, the operator defined by~\eqref{eq:SylpropsSol} is in $\Lin(W_\ga,X_e)$. Finally, we have from~\cite[Lem. 3.4]{PauPoh10} that since $\Sigma$ is in $\Lin(W_\ga,X_e)$, the condition~\eqref{eq:seriesass} is satisfied.

  To prove (d), let $k\in\Z$. If~\eqref{eq:Sylprops} has a solution $\Sigma\in \Lin(W_\ga,X_e)$, then clearly $\Sigma P_k\in \Lin(W,X_e)$ (since $\ran(P_k)\subset W_\ga$). For all $v\in \Dom(S)$ we have $P_kv\in W_{\ga+1}$ and
  \eq{
  \Sigma P_k S v  = \Sigma S P_k v = A_e \Sigma P_k v + B_e P_k v.
  }
  This concludes that $\Sigma P_k$ is a solution of the Sylvester equation $\Sigma_k S = A_e \Sigma_k + B_e P_k $. 
  The uniqueness of the solution follows from part (a) when we change $B_e$ to $B_eP_k$.
\end{proof}

\section{Characterizing Robustness with Respect to Given Perturbations}
\label{sec:RedIMs}

In this section we present a way of testing the robustness of a controller with respect to given perturbations. The following theorem extends the results presented in~\cite{PauPoh13a}, where the system had bounded input and output operators, the exosystem was finite-dimensional, and the closed-loop system was exponentially stable. Theorem~\ref{thm:Robchareqns} and its corollaries will also be instrumental in the proofs of the results presented in Sections~\ref{sec:pcopy} and~\ref{sec:Gconds}.

\begin{theorem}
  \label{thm:Robchareqns}
  A controller $(\mc{G}_1,\mc{G}_2,K)$ solving the output regulation problem is robust with respect to given perturbations $(\tilde{A},\tilde{B},\tilde{C},\tilde{D},\tilde{E},\tilde{F})\in \Ops$
  if and only if  the equations
  \begin{subequations}
    \label{eq:Robchar}
    \eqn{
    \tilde{\mathbb{P}}(i\gw_k)K \vec{z}^k &= -  \tilde{C} \mathbb{R}(i\gw_k, \tilde{A}) \tilde{E}\Phi_k -\tilde{F}\Phi_k   \label{eq:RobcharPK}\\[1ex]
    \JG \vec{z}^k &= 0 \label{eq:RobcharJG} 
    }
  \end{subequations}
  have a solution $\vec{z}^k = (z_{n_k}^k, \ldots, z_1^k)^T\in \Dom(\mc{G}_1)^{n_k}$ for all $k\in \Z$. Moreover, for every $k\in \Z$ the solution of~\eqref{eq:Robchar} is unique.
\end{theorem}

The proof of the theorem is based on the following lemma and certain properties of the regulator equations.

\begin{lemma}
  \label{lem:RORP} 
  Assume the controller\/ $(\mc{G}_1,\mc{G}_2,K)$ solves the output regulation problem on $W_\ga$. The controller is robust with respect to perturbations $(\tilde{A},\tilde{B},\tilde{C},\tilde{D},\tilde{E},\tilde{F})\in \Ops$
  if and only if
\ieq{
\tilde{C}_e \tilde{\Sigma} + \tilde{D}_e = 0
}
on $W_{\ga+1}$.
\end{lemma}

\begin{proof}
  Since the controller solves the output regulation problem on $W_\ga$, it remains to verify the third part of the robust output regulation problem. This part requires that the controller solves the output regulation problem for the perturbed operators $(\tilde{A},\tilde{B},\tilde{C},\tilde{D},\tilde{E},\tilde{F})$. However, since $\tilde{A}_e$ generates a strongly stable semigroup and since $\tilde{\Sigma} S=\tilde{A}_e\tilde{\Sigma}+\tilde{B}_e$ has a solution $\tilde{\Sigma}\in \Lin(W_\ga,X_e)$, we have from Theorem~\ref{thm:ORP} that this is true if and only if 
  $\tilde{C}_e \tilde{\Sigma} + \tilde{D}_e = 0$ is satisfied on $W_{\ga+1}$.
\end{proof}

The proof of Lemma~\ref{lem:Sylsplit} is presented in Appendix~\ref{sec:AppA}.

\begin{lemma}
  \label{lem:Sylsplit}
Assume $(\tilde{A},\tilde{B},\tilde{C},\tilde{D},\tilde{E},\tilde{F})$ satisfy parts \textup{(a)} and \textup{(b)} of Assumption~\textup{\ref{ass:pertclass}} and let $k\in\Z$.
For an operator $\Sigma = (\Pi, \Gamma)^T \in \Lin(W_\ga,X_e)$
the following are equivalent.
\begin{itemize}
  \item[\textup{(a)}] The operator $\Sigma$ satisfies $\ran(\Sigma P_k)\subset \Dom(\tilde{A}_e)$ and  $\Sigma P_k S = \tilde{A}_e \Sigma P_k + \tilde{B}_e P_k$
\item[\textup{(b)}] The operator $\Sigma$ satisfies
  $\ran(\Sigma P_k) \subset \Dom(\tilde{C})\times \Dom(\mc{G}_1)$ and
\begin{subequations}
        \label{eq:Sylsplit}
      \eqn{
        \label{eq:Sylsplit1}
      \JG  \Gamma \Phi_k &= \mc{G}_2\left(\tilde{\mathbb{P}}(i\gw_k)K \Gamma \Phi_k +  \tilde{C} \mathbb{R}(i\gw_k,\tilde{A}) \tilde{E}\Phi_k + \tilde{F}\Phi_k \right)\\
        \label{eq:Sylsplit2}
      \Pi \Phi_k &=
      \mathbb{R}(i\gw_k,\tilde{A}_{-1})\left( \tilde{B}K\Gamma \Phi_k + \tilde{E}\Phi_k\right).
      }
\end{subequations}
  \end{itemize}
  If $\Sigma = (\Pi,\Gamma)^T$ satisfies one of the above conditions, then
  \eqn{
  \label{eq:Sylsplit3}
  \tilde{C}_e\Sigma \Phi_k + \tilde{D}_e \Phi_k &= \tilde{\mathbb{P}}(i\gw_k )K\Gamma \Phi_k + \tilde{C} \mathbb{R}(i\gw_k,\tilde{A})\tilde{E}\Phi_k+  \tilde{F}\Phi_k.
  }

Moreover, if $(\tilde{A},\tilde{B},\tilde{C},\tilde{D},\tilde{E},\tilde{F})\in \Ops$, then for an operator $\Sigma : \Dom(\Sigma)\subset W\rightarrow X_e$
the following are equivalent.
\begin{itemize}
  \item[\textup{(c)}] The operator $\Sigma$ (or its extension) satisfies $\Sigma \in \Lin(W_\ga,X_e)$ and $\Sigma(W_{\ga+1})\subset \Dom(\tilde{A}_e)$, and it is a solution of the Sylvester equation $\Sigma S = \tilde{A}_e \Sigma + \tilde{B}_e $
\item[\textup{(d)}] The operator $\Sigma$ satisfies
  $\ran(\Sigma P_k) \subset \Dom(\tilde{C})\times \Dom(\mc{G}_1)$ and~\eqref{eq:Sylsplit} for all $k\in\Z$.
  \end{itemize}
  If one of the above conditions is satisfied, then~\eqref{eq:Sylsplit3} is satisfied for all $k\in\Z$.
\end{lemma}

The uniqueness of the solution of the Sylvester equation and Lemma~\ref{lem:Sylsplit} imply the following.

\begin{lemma}
  \label{lem:RORPcharunique}
Assume $(\tilde{A},\tilde{B},\tilde{C},\tilde{D},\tilde{E},\tilde{F})$ satisfy parts \textup{(a)} and \textup{(b)} of Assumption~\textup{\ref{ass:pertclass}} and let $k\in\Z$. If the equation
  \eqn{
  \label{eq:RORPcharunique}
  \JG \vec{z}^k = \mc{G}_2 \left( \tilde{\mathbb{P}}(i\gw_k)K\vec{z}^k + \tilde{C}\mathbb{R}(i\gw_k,\tilde{A})\tilde{E}\Phi_k + \tilde{F}\Phi_k \right)
  }
  has a solution $\vec{z}^k\in \Dom(\mc{G}_1)^{n_k}$, then $\Sigma_k S = \tilde{A}_e\Sigma_k + \tilde{B}_eP_k$ has a solution $\Sigma_k\in \Lin(W,X_e)$.

  On the other hand, if $(\tilde{A},\tilde{B},\tilde{C},\tilde{D},\tilde{E},\tilde{F})\in \Ops$, then for every $k\in\Z$ the equation~\eqref{eq:RORPcharunique} 
  has a unique solution $\vec{z}^k\in \Dom(\mc{G}_1)^{n_k}$. 
\end{lemma}

\begin{proof}
  To prove the first part of the lemma, let $k\in\Z$ and let $\vec{z}^k\in \Dom(\mc{G}_1)^{n_k}$ be a solution of~\eqref{eq:RORPcharunique}. Define  $\Pi\in \Lin(W,X)$, $\Gamma\in \Lin(W,Z)$ and $\Sigma = (\Pi,\Gamma)^T$ by
  \eqn{
  \label{eq:RORPcharuniquePGdef}
  \Gamma &=  \sum_{l=1}^{n_k}\iprod{\cdot}{\phi_k^l} z_l^k , \quad
  \Pi =  \sum_{l=1}^{n_k}\iprod{\cdot}{\phi_k^l} \sum_{j=0}^{l-1} (-1)^j R(i\gw_k, \tilde{A}_{-1} )^{j+1} \left( \tilde{B} K z_{l-j}^k + \tilde{E}\phi_k^{l-j} \right) . 
  \hspace{-1ex}
  }
The definitions imply that $\Pi\phi_k^l\in \Dom(\tilde{C})$ and $\Gamma \phi_k^l\in \Dom(\mc{G}_1)$ for all $l\in \List{n_k}$, and thus $\ran(\Sigma P_k) \subset \Dom(\tilde{C})\times \Dom(\mc{G}_1)$.
  For $l\in \List{n_k}$ we have
  $\Gamma \phi_k^l = z_l^k$, which together with the definition of $\Pi$ shows that 
  \eq{
  \Pi \Phi_k &=
  \mathbb{R}(i\gw_k,\tilde{A}_{-1})\left( \tilde{B}K\Gamma \Phi_k + \tilde{E}\Phi_k\right).
  }
  Furthermore, since $\Gamma \Phi_k = \vec{z}^k$, equation~\eqref{eq:RORPcharunique} implies
\eq{
\JG \Gamma \Phi_k 
=\JG \vec{z}^k 
&= \mc{G}_2 \left( \tilde{\mathbb{P}}(i\gw_k) K \vec{z}^k +  \tilde{C} \mathbb{R}(i\gw_k,\tilde{A}) E\Phi_k + F\Phi_k \right)\\
&= \mc{G}_2 \left( \tilde{\mathbb{P}}(i\gw_k) K \Gamma \Phi_k +  \tilde{C} \mathbb{R}(i\gw_k,\tilde{A}) E\Phi_k + F\Phi_k \right).
  } 
  This concludes that $\Sigma$ satisfies~\eqref{eq:Sylsplit}, and thus we have from Lemma~\ref{lem:Sylsplit} that $\Sigma P_k = \Sigma$ is a solution of the Sylvester equation $\Sigma_k S = \tilde{A}_e \Sigma_k + \tilde{B}_eP_k$.

  To prove the second part, assume $(\tilde{A},\tilde{B},\tilde{C},\tilde{D},\tilde{E},\tilde{F})\in \Ops$ and let $k\in\Z$. We have from Assumption~\ref{ass:pertclass} that the Sylvester equation $\Sigma S = \tilde{A}_e \Sigma + \tilde{B}_e$ has a solution $\Sigma\in \Lin(W_\ga,X_e)$. 
If we let $k\in\Z$ and denote $\vec{z}^k = \Gamma \Phi_k\in \Dom(\mc{G}_1)^{n_k}$, then we have from~\eqref{eq:Sylsplit1} and Lemma~\ref{lem:Sylsplit} that $\vec{z}^k$ is the solution of~\eqref{eq:RORPcharunique}.

To prove the uniqueness of the solution, let $\vec{z}^k,\tilde{\vec{z}}^k\in \Dom(\mc{G}_1)^{n_k}$ be two solutions of~\eqref{eq:RORPcharunique}. We can now use formulas~\eqref{eq:RORPcharuniquePGdef} to define operators $\Sigma = (\Pi,\Gamma)^T$ and $\tilde{\Sigma}=(\tilde{\Pi}, \tilde{\Gamma})^T$ corresponding to $\vec{z}^k$ and $\tilde{\vec{z}}^k$, respectively. As in the beginning of this proof, we get that $\Sigma$ and $\tilde{\Sigma}$ are solutions of the Sylvester equation $\Sigma_k = \tilde{A}_e \Sigma_k + \tilde{B}_eP_k$.  However, by Theorem~\ref{thm:Sylprop} the solution of this equation is unique, and we must thus have $\Sigma_k = \tilde{\Sigma}_k$ and, in particular, $\Gamma_k = \tilde{\Gamma}_k$. From the definitions of these operators it is clear that this is only possible if $\vec{z}^k = \tilde{\vec{z}}^k$. This concludes that the solution of~\eqref{eq:RORPcharunique} is unique.
\end{proof}

  \textit{Proof of Theorem~\textup{\ref{thm:Robchareqns}}}.
Let $(\tilde{A},\tilde{B},\tilde{C},\tilde{D},\tilde{E},\tilde{F})\in \Ops$.

  We will first show that robustness of a controller with respect to the given perturbations implies that the equations~\eqref{eq:Robchar} have solutions for all $k\in \Z$.
  The robustness of the controller together with Lemma~\ref{lem:RORP} implies that 
  \begin{subequations}
    \label{eq:RR}
    \eqn{
    \label{eq:RR1}
    \tilde{\Sigma} S &= \tilde{A}_e \tilde{\Sigma} + \tilde{B}_e\\
    \label{eq:RR2}
    0&=\tilde{C}_e \tilde{\Sigma } + \tilde{D}_e
    }
  \end{subequations}
  have a solution $\tilde{\Sigma} = (\tilde{\Pi}, \tilde{\Gamma})^T\in \Lin(W_\ga,X_e)$.
  Let $k\in \Z$.
  We now have from~\eqref{eq:Sylsplit1} and~\eqref{eq:Sylsplit3} in Lemma~\ref{lem:Sylsplit} that the perturbed regulator equations~\eqref{eq:RR} in particular imply
  \eq{
  \JG  \tilde{\Gamma} \Phi_k &= \mc{G}_2\left( \tilde{\mathbb{P}}(i\gw_k)K \tilde{\Gamma} \Phi_k +  \tilde{C} \mathbb{R}(i\gw_k,\tilde{A}) \tilde{E}\Phi_k + \tilde{F}\Phi_k \right)\\
  0 &= \tilde{\mathbb{P}}(i\gw_k )K\tilde{\Gamma} \Phi_k + \tilde{C} \mathbb{R}(i\gw_k,\tilde{A})\tilde{E}\Phi_k+  \tilde{F}\Phi_k .
  }
  If we choose $\vec{z}^k = \Gamma \Phi_k \in \Dom(\mc{G}_1)^{n_k}$, then~\eqref{eq:RobcharPK} follows immediately from the second equation. Furthermore, substituting the second equation into the right-hand side of the first further concludes $\JG \vec{z}^k = 0$, and thus $\vec{z}^k$ is the solution of the equations~\eqref{eq:Robchar}. Since $k\in \Z$ was arbitrary, this concludes the first part of the proof.

  Now assume that for all $k\in \Z$ equations~\eqref{eq:Robchar} have solutions $\vec{z}^k = (z_{n_k}^k,\ldots, z_1^k)^T\in \Dom(\mc{G}_1)^{n_k}$.
  Define operators $\Pi: \Dom(\Pi)\subset  W_\ga \rightarrow X$, $\Gamma : \Dom(\Gamma)\subset  W_\ga \rightarrow Z$, and $\Sigma : \Dom(\Sigma)\subset W_\ga \rightarrow X_e$ by
  \eq{
  \Gamma = \sum_{k\in\Z} \sum_{l=1}^{n_k}\iprod{\cdot}{\phi_k^l} z_l^k , \quad 
\Pi = \sum_{k\in\Z} \sum_{l=1}^{n_k}\iprod{\cdot}{\phi_k^l} \sum_{j=0}^{l-1} (-1)^j R(i\gw_k, \tilde{A}_{-1} )^{j+1} \left( \tilde{B} K z_{l-j}^k + \tilde{E}\phi_k^{l-j} \right)  
  } 
  and $\Sigma = (\Pi,\Gamma)^T$. 
  We will show that $\Sigma$ is the solution of the perturbed Sylvester equation~\eqref{eq:RR1}, and that it satisfies the regulation constraint~\eqref{eq:RR2}.

  Let $k\in\Z$.  
  For all $l\in \List{n_k}$ we have
  $\Gamma \phi_k^l = z_l^k \in \Dom(\mc{G}_1)$, which together with the definition of $\Pi$ implies that $\ran(\Pi P_k)\subset \Dom(\tilde{C})$ and that~\eqref{eq:Sylsplit2} is satisfied.
Furthermore, since $\Gamma \Phi_k = \vec{z}^k$, we have from~\eqref{eq:Robchar} that
\eq{
\JG \Gamma \Phi_k 
=\JG \vec{z}^k = 0 
&= \mc{G}_2 \left( \tilde{\mathbb{P}}(i\gw_k) K \vec{z}^k +  \tilde{C} \mathbb{R}(i\gw_k,\tilde{A}) E\Phi_k + F\Phi_k \right)\\
&= \mc{G}_2 \left( \tilde{\mathbb{P}}(i\gw_k) K \Gamma \Phi_k +  \tilde{C} \mathbb{R}(i\gw_k,\tilde{A}) E\Phi_k + F\Phi_k \right),
  } 
  which is precisely~\eqref{eq:Sylsplit1}. 
  Since $(\tilde{A},\tilde{B},\tilde{C},\tilde{D},\tilde{E},\tilde{F})\in \Ops$, we now have from the second part of Lemma~\ref{lem:Sylsplit} that $\Sigma \in \Lin(W_\ga,X_e)$ and it is the solution of the Sylvester equation $\Sigma S = \tilde{A}_e \Sigma + \tilde{B}_e$.
 Finally,  Lemma~\ref{lem:Sylsplit} and equation~\eqref{eq:RobcharPK} imply that 
  \eq{
  \tilde{C}_e\Sigma \Phi_k + \tilde{D}_e \Phi_k 
  &= \tilde{\mathbb{P}}(i\gw_k )K\Gamma \Phi_k + \tilde{C} \mathbb{R}(i\gw_k,\tilde{A})\tilde{E}\Phi_k+  \tilde{F}\Phi_k \\
  &= \tilde{\mathbb{P}}(i\gw_k )K \vec{z}^k + \tilde{C} \mathbb{R}(i\gw_k,\tilde{A})\tilde{E}\Phi_k+  \tilde{F}\Phi_k 
  =0.
  }
  Since $k\in\Z$ was arbitrary and $\tilde{C}_e\Sigma + \tilde{D}_e\in \Lin(W_{\ga+1},X_e)$ by Lemma~\ref{lem:CSigmarelbdd}, we have $\tilde{C}_e\Sigma + \tilde{D}_e=0$ on $W_{\ga+1}$.
  Thus $\Sigma$ is a solution of the perturbed regulator equations, and by Lemma~\ref{lem:RORP} the controller is robust with respect to the perturbations $(\tilde{A},\tilde{B},\tilde{C},\tilde{D},\tilde{E},\tilde{F})$.

  It remains to prove the uniqueness of the solution of~\eqref{eq:Robchar}. If $\vec{z}^k$ is the solution of the equations~\eqref{eq:Robchar}, then it is also clearly a solution of the equation~\eqref{eq:RORPcharunique}.
  By Lemma~\ref{lem:RORPcharunique} the solution of this equation is unique, and therefore the same is also true for the solution of~\eqref{eq:Robchar}.
  $\Box$

From Theorem~\ref{thm:Robchareqns} and Lemma~\ref{lem:RORPcharunique} we get the following corollary. This will be helpful in characterizing the robustness of a controller through the \Gconds.

\begin{corollary}
  \label{cor:IMSindfreq}
  Assume $(\tilde{A},\tilde{B},\tilde{C},\tilde{D},\tilde{E},\tilde{F})\in \Ops$. The controller is robust with respect to the perturbations if and only if for every $k\in \Z$ the unique solution $\vec{z}^k\in \Dom(\mc{G}_1)^{n_k}$ of the equation
\begin{subequations}
  \label{eq:IMSfreq}
  \eqn{
  \label{eq:IMSfreq1}
  \JG \vec{z}^k = \mc{G}_2 \left( \tilde{\mathbb{P}}(i\gw_k)K\vec{z}^k + \tilde{C}\mathbb{R}(i\gw_k,\tilde{A})\tilde{E}\Phi_k + \tilde{F}\Phi_k \right)
  }
  satisfies 
  \eqn{
  \label{eq:IMSfreq2}
  \tilde{\mathbb{P}}(i\gw_k)K\vec{z}^k + \tilde{C}\mathbb{R}(i\gw_k,\tilde{A})\tilde{E}\Phi_k + \tilde{F}\Phi_k =0.
  }
\end{subequations}
\end{corollary}

\section{The $p$-Copy Internal Model Principle}
\label{sec:pcopy}

In this section we show that a controller stabilizing the closed-loop system solves the robust output regulation problem if and only if it incorporates a p-copy internal model of the exosystem. The `$p$' in the term refers to the dimension of the output space, i.e., $p=\dim Y$. The significance of $p$ is that the classical definition of the finite-dimensional internal model states roughly that ``for any Jordan block of $S$ associated to an eigenvalue $s$, the matrix $\mc{G}_1$ must have at least $p$ Jordan blocks of greater or equal size associated to $s$''.
For infinite-dimensional feedback controllers the p-copy internal model can be defined as shown below~\cite{PauPoh10}. The definition of the p-copy is meaningful only in the case of a finite-dimensional output space $Y$.

\begin{definition}[The p-copy internal model]
  \label{def:pcopy}
Assume $\dim Y<\infty$. A controller~$(\mc{G}_1,\mc{G}_2,K)$ is said to\/ {\em incorporate a p-copy internal model} of the exosystem $S$ if for all $k\in\Z$ we 
have
\eq{
\dim\ker(i\gw_k-\mc{G}_1)\geq\dim\, Y
}
and\/ $\mc{G}_1$ has at least\/ $\dim Y$ independent Jordan chains of length greater than or equal to~$n_k$ associated to the eigenvalue $i\gw_k$.
\end{definition}

The following theorem is the main result of this section.

\begin{theorem}
  \label{thm:IMP}
  Assume that $\dim Y<\infty$, the controller $(\mc{G}_1,\mc{G}_2,K)$ stabilizes the closed-loop system strongly, $i\gw_k\in \rho(A_e)$ for all $k\in\Z$, and the Sylvester equation $\Sigma S = A_e\Sigma + B_e$ has a solution $\Sigma\in \Lin(W_\ga,X_e)$. Then the controller solves the robust output regulation problem on $W_\ga$ if and only if it incorporates a p-copy internal model of the exosystem.
\end{theorem}

As a by-product of the proof of Theorem~\ref{thm:IMP}, we obtain a new way of defining an ``internal model'' of the exosystem for infinite-dimensional controllers. This definition can be given in a compact form using the properties of the operator
  \eqn{
  \label{eq:pcopyvsPKmatinv}
  (\tilde{\mathbb{P}}(i\gw_k)K)\vert_{\ker(\JG)}: \ker(\JG)\subset Z^{n_k}\rightarrow Y^{n_k},
  }
  i.e., the restriction of the operator $\tilde{\mathbb{P}}(i\gw_k)K$ to the subspace $\ker(\JG)$.
  The following theorem shows that the invertibility of the above operator is equivalent to the controller incorporating an internal model of the exosystem in the sense of Definition~\ref{def:pcopy}. The theorem generalizes the results in~\cite[Sec.~6]{PauPoh10}, where it was shown that for a diagonal exosystem the invertibility of the operators $(P(i\gw_k)K)\vert_{\ker(i\gw_k-\mc{G}_1)}$ for all frequencies $i\gw_k$ is equivalent to the controller incorporating an internal model.

\begin{theorem}
  \label{thm:pcopyvsPKmatinv}
  Assume $\dim Y <\infty$.
 
  If there exist $(\tilde{A},\tilde{B}, \tilde{C}, \tilde{D},\tilde{E},\tilde{F})\in \Ops$ such that the operator in~\eqref{eq:pcopyvsPKmatinv}
  is surjective for all $k\in \Z$, then the controller incorporates a p-copy internal model of the exosystem.

  Conversely, if the controller incorporates a p-copy internal model of the exosystem and if $(\tilde{A},\tilde{B}, \tilde{C}, \tilde{D},\tilde{E},\tilde{F})\in \Ops$, then the operator~\eqref{eq:pcopyvsPKmatinv} is boundedly invertible for all $k\in \Z$.
\end{theorem}

\begin{remark}
  The conclusions of Theorem~\textup{\ref{thm:pcopyvsPKmatinv}} are in particular true for the unperturbed operators $(A,B,C,D,E,F)$.
\end{remark}

The proof of Theorem~\ref{thm:pcopyvsPKmatinv} is based on the following four lemmas. Lemma~\ref{lem:PKinj} was first introduced 
in~\cite{PauPoh10}
for the transfer function $P(\gl)$ of the unperturbed plant.

\begin{lemma}
  \label{lem:PKinj}
  Let
  $k\in \Z$.
If $i\gw_k \in \rho(\tilde{A})$ and $i\gw_k\notin \gs_p(\tilde{A}_e)$,
  then $(\tilde{P}(i\gw_k)K)\vert_{\ker(i\gw_k-\mc{G}_1)}$ is injective.
\end{lemma}

\begin{proof}
  Let  $z\in\ker(i\gw_k -\mc{G}_1)$ be such that $\tilde{P}(i\gw_k)Kz=0$. 
  Since $i\gw_k \in \rho(\tilde{A}) = \rho(\tilde{A}_{-1})$, 
we can choose \mbox{$x=R(i\gw_k,\tilde{A}_{-1})\tilde{B}Kz\in\Dom(\tilde{C})$}. On $X_{-1}\times Z$ we have
\eq{
\MoveEqLeft\pmat{(i\gw_k -\tilde{A}_{-1})x-\tilde{B}Kz\\-\mc{G}_2\tilde{C}x+(i\gw_k-\mc{G}_1)z -\mc{G}_2\tilde{D}Kz}
=  \pmat{\tilde{B}Kz-\tilde{B}Kz\\-\mc{G}_2(\tilde{C}R(i\gw_k,\tilde{A}_{-1})\tilde{B}+\tilde{D})Kz+(i\gw_k-\mc{G}_1)z} \\[1ex]
&= \pmat{0\\- \mc{G}_2 \tilde{P}(i\gw_k)K z}
=\pmat{0\\0} \in X\times Z.
}
This shows that $(x,z)^T\in \Dom(i\gw_k - \tilde{A}_e)$ and $(i\gw_k - \tilde{A}_e)\left( {x\atop z} \right)=\left( {0\atop 0} \right)$.
Since $i\gw_k\notin\gs_p(\tilde{A}_e)$, we know that $i\gw_k-\tilde{A}_e$ is injective. 
This in particular implies $z=0$, which concludes that the restriction of $\tilde{P}(i\gw_k)K$ to \mbox{$\ker(i\gw_k-\mc{G}_1)$} is injective. 
\end{proof}

\begin{lemma}
  \label{lem:PKmatinj}
  If $(\tilde{A},\tilde{B}, \tilde{C}, \tilde{D}, \tilde{E}, \tilde{F})\in \Ops$,
 then $(\tilde{\mathbb{P}}(i\gw_k)K)\vert_{\ker(\JG)}$ is injective for all $k\in\Z$.
\end{lemma}

\begin{proof}
  Let $k\in\Z$. We have from Assumption~\ref{ass:pertclass} that $i\gw_k\in \rho(\tilde{A})$, and since $\tilde{A}_e$ generates a strongly stable semigroup, we must have $\gs_p(\tilde{A}_e)\cap i\R=\varnothing$. Therefore the conditions of Lemma~\ref{lem:PKinj} are satisfied and the operator $(\tilde{P}(i\gw_k)K)\vert_{\ker(i\gw_k - \mc{G}_1)}$ is injective.

  Let $\vec{z} = (z_{n_k},\ldots,z_1)^T\in \ker(\JG)\subset  \Dom(\mc{G}_1)^{n_k}$ be such that
  $\tilde{\mathbb{P}}(i\gw_k) K \vec{z}  = 0 $. If we denote $\tilde{P}_l = (-1)^l \tilde{C} R(i\gw_k, \tilde{A}_{-1})^{l+1} \tilde{B} $ for $l\in \List{n_k-1}$, then the equation can be written as
  \eq{
\pmatsmall{\tilde{P}(i\gw_k)K &  \tilde{P}_1 K & \cdots &  \tilde{P}_{n_k-1} K \\
&\ddots\\
&& \tilde{P}(i\gw_k)K & \tilde{P}_1 K \\
&&& \tilde{P}(i\gw_k)K}
\pmat{z_{n_k}\\\vdots \\z_1}
= \pmat{0\\\vdots \\0}.
}
Since $z_1\in \ker(i\gw_k - \mc{G}_1)$ and since $(\tilde{P}(i\gw_k)K)\vert_{\ker(i\gw_k - \mc{G}_1)}$ is injective by Lemma~\ref{lem:PKinj}, the last line implies $z_1=0$. Since $\JG \vec{z} = 0$, we also have $(\mc{G}_1 - i\gw_k ) z_2 = z_1=0$, and thus $z_2 \in \ker(i\gw_k - \mc{G}_1)$.
Substituting $z_1= 0$ to the second last line of the matrix equation becomes $P(i\gw_k)K z_2=0$, and the injectivity of $(P(i\gw_k)K)\vert_{\ker(i\gw_k - \mc{G}_1)}$ implies $z_2=0$. 

These steps can be repeated until we have reached $z_1 = \cdots = z_{n_k-1}=0$, and $(\mc{G}_1 - i\gw_k)z_{n_k}=z_{n_k-1}=0$ shows that $z_{n_k}\in \ker(i\gw_k - \mc{G}_1)$. Substituting these to the top line of the matrix equation we get $P(i\gw_k)K z_{n_k} = 0$, which in turn implies $z_{n_k}=0$ due to the injectivity of the operator $(P(i\gw_k)K)\vert_{\ker(i\gw_k - \mc{G}_1)}$. This finally concludes $\vec{z}=0$.
\end{proof}

\begin{lemma}
  \label{lem:PKmatsurjtopcopy}
  Assume $\dim Y <\infty$.
  If $(\tilde{A},\tilde{B}, \tilde{C}, \tilde{D}, \tilde{E}, \tilde{F})\in \Ops$
  are such that the operator
$(\tilde{\mathbb{P}}(i\gw_k)K)\vert_{\ker(\JG)}$ 
  is surjective for all $k\in \Z$, then the controller incorporates a p-copy internal model of the exosystem.
\end{lemma}

\begin{proof}
  Let $p=\dim Y$ and $k\in \Z$. By construction, if $\vec{z} = (z_{n_k},\ldots,z_1)^T \in \ker(\JG)$ is such that $z_1\neq 0$, then $(z_j)_{j=1}^{n_k}$ is a Jordan chain of $\mc{G}_1$ associated to the eigenvalue $i\gw_k$.
  Let $\set{e_l}_{l=1}^p \subset Y=\C^p$ be the natural basis vectors of $\C^p$. Then by surjectivity of $(\tilde{\mathbb{P}}(i\gw_k)K)\vert_{\ker(\JG)}$ there exist $\set{\vec{z}^l}_{l=1}^p \subset \ker(\JG)$ such that 
  \eq{
  \pmatsmall{\tilde{P}(i\gw_k)K &  \tilde{P}_1 K & \cdots &  \tilde{P}_{n_k-1} K \\
&\ddots\\
&& \tilde{P}(i\gw_k)K & \tilde{P}_1 K \\
&&& \tilde{P}(i\gw_k)K}
\pmat{z_{n_k}^l\\\vdots \\z_1^l}
= \pmat{0\\\vdots \\e_l}
  }
  for all $l\in \List{p}$.
  The bottom lines of the equations show that $\tilde{P}(i\gw_k)K z_1^l = e_l$, and therefore the first elements $\set{z_1^l}_{l=1}^p$ must be linearly independent, because $\set{e_l}_{l=1}^p$ are linearly independent. This concludes that $\mc{G}_1$ has $p$ independent Jordan chains $\set{z_j^l}_{j=1}^{n_k}$ of lengths $n_k$ associated to the eigenvalue $i\gw_k$. Since $k\in \Z$ was arbitrary, this concludes the proof.
\end{proof}

\begin{lemma}
  \label{lem:PKmatsurj}
  Assume $\dim Y <\infty$.
  If the controller incorporates a p-copy internal model of the exosystem and if
  $(\tilde{A},\tilde{B}, \tilde{C}, \tilde{D}, \tilde{E}, \tilde{F})\in \Ops$,
  then the operator
$(\tilde{\mathbb{P}}(i\gw_k)K)\vert_{\ker(\JG)}$ 
  is surjective for all $k\in \Z$.
\end{lemma}

\begin{proof}
Let $p=\dim Y$, $(\tilde{A},\tilde{B},\tilde{C},\tilde{D}, \tilde{E}, \tilde{F})\in \Ops$, and $k\in \Z$.
Then $i\gw_k \in \rho(\tilde{A})$ by Assumption~\ref{ass:pertclass}, and since $\tilde{A}_e$ generates a strongly stable semigroup, we have $i\gw_k\notin \gs_p(\tilde{A}_e)$. Thus we have from Lemma~\ref{lem:PKinj} that $(\tilde{P}(i\gw_k)K)\vert_{\ker(i\gw_k-\mc{G}_1)}$ is injective, and since $\dim \ker(i\gw_k - \mc{G}_1) \geq p$ due to the p-copy internal model, it also surjective. Actually, the Rank Nullity Theorem~\cite[Thm. 4.7.7]{NaySel82book} together with the invertibility of $(\tilde{P}(i\gw_k)K)\vert_{\ker(i\gw_k-\mc{G}_1)}$ implies 
\eq{
\dim \ker(i\gw_k - \mc{G}_1)
&=\dim \ran \left( (\tilde{P}(i\gw_k)K)\vert_{\ker(i\gw_k-\mc{G}_1)} \right)
+\dim \ker \left( (\tilde{P}(i\gw_k)K)\vert_{\ker(i\gw_k-\mc{G}_1)} \right)\\
&=\dim \ran \left( (\tilde{P}(i\gw_k)K)\vert_{\ker(i\gw_k-\mc{G}_1)} \right)
=\dim Y
=p.
}

We will show that for any $\vec{y} = (y_{n_k},\ldots,y_1)^T\in Y^{n_k}$ we can choose an element $\vec{z} = (z_{n_k},\ldots,z_1)^T\in \ker(\JG)\subset  \Dom(\mc{G}_1)^{n_k}$ such that
$\tilde{\mathbb{P}}(i\gw_k) K \vec{z}  = \vec{y} $. Denote $\tilde{P}_l = (-1)^l \tilde{C} R(i\gw_k, \tilde{A}_{-1})^{l+1} \tilde{B} $ for $l\in \List{n_k-1}$. The equation can be written as
\eqn{
\label{eq:PKmatsurj}
\pmatsmall{\tilde{P}(i\gw_k)K &  \tilde{P}_1 K & \cdots &  \tilde{P}_{n_k-1} K \\
&\ddots\\
&& \tilde{P}(i\gw_k)K & \tilde{P}_1 K \\
&&& \tilde{P}(i\gw_k)K}
\pmat{z_{n_k}\\\vdots \\z_1}
= \pmat{y_{n_k}\\\vdots \\y_1}.
}
It was shown in~\cite[Lem. 6.8]{PauPoh10} that since the controller incorporates a p-copy internal model of the exosystem and since
$\dim \ker(i\gw_k - \mc{G}_1)=p$, we have $\ker(i\gw_k - \mc{G}_1)^{n_k-1}\subset \ran(i\gw_k - \mc{G}_1)$.

Since $(\tilde{P}(i\gw_k)K)\vert_{\ker(i\gw_k-\mc{G}_1)}$ is surjective, we can choose $z_1\in \ker(i\gw_k - \mc{G}_1)$ in such a way that $\tilde{P}(i\gw_k) K z_1 = y_1$. This shows that the bottom line of equation~\eqref{eq:PKmatsurj} is satisfied. If $n_k=1$, the proof is complete. Otherwise we continue as follows.

Since $\ker(i\gw_k - \mc{G}_1)\subset \ker(i\gw_k - \mc{G}_1)^{n_k-1}\subset \ran(i\gw_k - \mc{G}_1)$, we can choose $\tilde{z}_2 \in \Dom(\mc{G}_1)$ such that $(\mc{G}_1 - i\gw_k) \tilde{z}_2 = z_1$. Now choose $\gd_2 \in \ker(i\gw_k - \mc{G}_1)$ in such a way that
\eq{
\tilde{P}(i\gw_k) K \gd_2 = y_2 - \tilde{P}(i\gw_k) K \tilde{z}_2 - \tilde{P}_1 K z_1 .
}
This is possible since $(\tilde{P}(i\gw_k)K)\vert_{\ker(i\gw_k- \mc{G}_1)}$ is surjective. 
If we choose $z_2 = \tilde{z}_2 + \gd_2$, then $(\mc{G}_1 - i\gw_k )z_2 = z_1$, and 
\eq{
\tilde{P}(i\gw_k) K z_2  + \tilde{P}_1  K z_1 = y_2.
}
This shows that the second last line of equation~\eqref{eq:PKmatsurj} is satisfied.

These same steps can be repeated until we have chosen $\set{z_l}_{l=1}^{n_k-1}$ in such a way that $n_k-1$ lines from the bottom of equation~\eqref{eq:PKmatsurj} are satisfied and $\set{z_l}_{l=1}^{n_k-1}$ is a Jordan chain of $\mc{G}_1$. Then, since $z_{n_k-1}\in \ker(i\gw_k - \mc{G}_1)^{n_k-1} \subset \ran(i\gw_k - \mc{G}_1)$, we can choose $\tilde{z}_{n_k} \in \Dom(\mc{G}_1)$ such that $(\mc{G}_1 - i\gw_k) \tilde{z}_{n_k} = z_{n_k-1}$. Now choose $\gd_{n_k} \in \ker(i\gw_k - \mc{G}_1)$ in such a way that
\eq{
\tilde{P}(i\gw_k) K \gd_{n_k} =  y_{n_k} - \tilde{P}(i\gw_k) K \tilde{z}_{n_k} - \sum_{l=1}^{n_k-1} \tilde{P}_l K z_{n_k-l} .
}
This is possible since $(\tilde{P}(i\gw_k)K)\vert_{\ker(i\gw_k- \mc{G}_1)}$ is surjective. 
If we choose $z_{n_k} = \tilde{z}_{n_k} + \gd_{n_k}$, then $(\mc{G}_1 - i\gw_k )z_{n_k} = z_{n_k-1}$, and 
\eq{
\tilde{P}(i\gw_k) K z_{n_k}  + \sum_{l=1}^{n_k-1} \tilde{P}_l K z_{n_k-l} = y_{n_k}.
}
This finally shows that the first line of equation~\eqref{eq:PKmatsurj} is satisfied. By construction we thus have $\tilde{P}(i\gw_k)K \vec{z} = \vec{y}$, and $\set{z_l}_{l=1}^{n_k}$ is a Jordan chain of $\mc{G}_1$ associated to $i\gw_k$, i.e., $\JG \vec{z} = 0$. This concludes the proof.  
\end{proof}

\textit{Proof of Theorem~\textup{\ref{thm:pcopyvsPKmatinv}}.}
  The first claim follows directly from Lemma~\ref{lem:PKmatsurjtopcopy}. The second claim follows from Lemmas~\ref{lem:PKmatinj} and~\ref{lem:PKmatsurj}.
  $\Box$

We can now use Theorem~\ref{thm:pcopyvsPKmatinv} to present a proof for the p-copy internal model principle in Theorem~\ref{thm:IMP}.

\textit{Proof of Theorem~\textup{\ref{thm:IMP}}.} 
 We begin by showing that a controller incorporating a p-copy internal model solves the robust output regulation problem.
  To this end, let $(\tilde{A},\tilde{B},\tilde{C},\tilde{D},\tilde{E},\tilde{F})\in \Ops$. 
  We have from Theorem~\ref{thm:pcopyvsPKmatinv} that  $(\tilde{\mathbb{P}}(i\gw_k) K)\vert_{\ker(\JG)}$ are invertible for all $k\in \Z$. This means in particular that for any $k\in \Z$ we can choose $\vec{z}^k\in \Dom(\mc{G}_1)^{n_k}$ in such a way that $\JG \vec{z}^k = 0$ and 
  \eq{
  \tilde{\mathbb{P}}(i\gw_k) K \vec{z}^k = - \tilde{C} \mathbb{R}(i\gw_k, \tilde{A})\tilde{E} \Phi_k  - \tilde{F}\Phi_k.
  }
  Therefore the equations~\eqref{eq:Robchar} have a solution for all $k\in \Z$, and Theorem~\ref{thm:Robchareqns} states that the controller is robust with respect to the given perturbations. Since the perturbations were arbitrary, this concludes the proof.

  Conversely, assume that the controller solves the robust output regulation problem. By Theorem~\ref{thm:pcopyvsPKmatinv} it is sufficient to show that for some perturbations in $\Ops$ the operator $(\tilde{\mathbb{P}}(i\gw_k))\vert_{\ker(\JG)}$ is surjective for all $k\in\Z$.  We leave the operators $(A,B,C,D)$ unperturbed and show that $(\mathbb{P}(i\gw_k))\vert_{\ker(\JG)}$ are surjective by choosing the perturbed operators $\tilde{E}$ and $\tilde{F}$ in a suitable way. The closed-loop system is strongly stable and parts (a) and (b) of Assumption~\ref{ass:pertclass} are satisfied. Let $k\in\Z$ be fixed. We have $i\gw_k\in \rho(A_e)$ by assumption. 
  Let $\vec{y} = (y_{n_k},\ldots,y_1)^T\in Y^{n_k}$, and choose $\tilde{E}=0\in \Lin(W,X_e)$ and $\tilde{F}= -\sum_{l=1}^{n_k} \iprod{\cdot}{\phi_k^l}y_l$. We then have $\tilde{F}\Phi_k = -\vec{y}$. Since $\tilde{B}_e\phi_{k'}^l= (\tilde{E}\phi_{k'}^l,\mc{G}_2 \tilde{F}\phi_{k'}^l)^T=0$ for any $k'\neq k$, the supremum in~\eqref{eq:seriesass} is clearly finite and we have from part (b) of Theorem~\ref{thm:Sylprop} that the Sylvester equation $\Sigma S = A_e \Sigma + \tilde{B}_e$ has a solution $\Sigma\in \Lin(W_\ga,X_e)$. This concludes that the perturbations satisfy $(A,B,C,D,\tilde{E},\tilde{F})\in \Ops$.

  Since the controller solves the robust output regulation problem and since we have $(A,B,C,D,\tilde{E},\tilde{F})\in \Ops$, Theorem~\ref{thm:Robchareqns} implies that there exists $\vec{z}^k\in \ker(\JG)$ such that
  \eq{
  \mathbb{P}(i\gw_k) K \vec{z}^k &= - C \mathbb{R}(i\gw_k, A)\tilde{E} \Phi_k  - \tilde{F}\Phi_k\\
  \Leftrightarrow \qquad \mathbb{P}(i\gw_k) K \vec{z}^k &=   \vec{y} .
  }
  Since $\vec{y}\in Y^{n_k}$ was arbitrary, this shows that $(\mathbb{P}(i\gw_k))\vert_{\ker(\JG)}$ is surjective. 
  The index $k\in\Z$ was arbitrary, and thus we have from Theorem~\ref{thm:pcopyvsPKmatinv} that the controller incorporates a p-copy internal model.
  $\Box$

\section{The $\mc{G}$-Conditions}
\label{sec:Gconds}

In this section we show that also the so-called \keyterm{\Gconds}~\cite{HamPoh10,PauPoh10} can be used in characterizing controllers that solve the robust output regulation problem. As we will see in Section~\ref{sec:infexoex}, one of the strengths of the \Gconds\ is that they are straightforward to verify for a certain type of triangular controllers. Moreover, this version of the internal model is meaningful also in the situation where the output space $Y$ is infinite-dimensional.

\begin{definition}[The $\mc{G}$-conditions]
  \label{def:Gconds}
  A controller\/ $(\mc{G}_1,\mc{G}_2,K)$ is said to satisfy the\/ {\em\Gconds} 
  if
\begin{subequations}
  \label{eq:Gconds}
\eqn{
\ran(i\gw_k-\mc{G}_1)\cap\ran(\mc{G}_2)&=\set{0} ~\quad\qquad\qquad \forall k\in\Z , \label{eq:Gconds1} 
\\
\ker(\mc{G}_2)&=\set{0} \label{eq:Gconds2}\\
\label{eq:Gconds3}
\ker(i\gw_k-\mc{G}_1)^{n_k-1}&\subset\ran(i\gw_k-\mc{G}_1) \qquad  \forall k\in\Z.
}
\end{subequations}
\end{definition}

The following theorem is the main result of this section. It was shown in~\cite[Lem. 5.7]{PauPoh10} that the condition $Z=\ran(i\gw_k - \mc{G}_1)+ \ran(\mc{G}_2)$ is in particular true if $i\gw_k\in \rho(A_e)$.

\begin{theorem}
  \label{thm:GcondIMP}
  Assume that the controller $(\mc{G}_1,\mc{G}_2,K)$ stabilizes the closed-loop system strongly and satisfies $Z=\ran(i\gw_k - \mc{G}_1)+ \ran(\mc{G}_2)$ for all $k\in\Z$, and the Sylvester equation $\Sigma S = A_e\Sigma + B_e$ has a solution $\Sigma\in \Lin(W_\ga,X_e)$. Then the controller solves the robust output regulation problem on $W_\ga$ if and only if it satisfies the \Gconds.
\end{theorem}

The proof of Theorem~\ref{thm:GcondIMP} is a direct consequence of the following four lemmas.

\begin{lemma}
\label{lem:RobtoGc1}
If the controller\/ $(\mc{G}_1,\mc{G}_2,K)$ solves the robust output regulation problem, then\/ \eqref{eq:Gconds1} is satisfied.
\end{lemma}

\begin{proof}
Let $k\in\Z$ and $w\in\ran(i\gw_k-\mc{G}_1)\cap\ran(\mc{G}_2)$. Then there 
exist $z\in\Dom(\mc{G}_1)$ and $y\in Y$ such that 
\vspace{-1.5ex}
\eq{
w=(i\gw_k-\mc{G}_1)z=\mc{G}_2y. 
}
Leave the operators $(A,B,C,D)$ unperturbed, and choose $\tilde{E}=0\in \Lin(W,X)$ and $\tilde{F} = \iprod{\cdot}{\phi_k^{n_k}} (y-P(i\gw_k)Kz)\in \Lin(W,Y)$. 
The operators $(A,B,C,D,\tilde{E},\tilde{F})$ satisfy the parts (a) and (b) of Assumption~\ref{ass:pertclass}.

Now $\tilde{F}\Phi_k=(y-P(i\gw_k)Kz,0,\ldots,0)^T\in Y^{n_k}$. For $\vec{z}^k = (z,0,\ldots,0)^T\in \Dom(\mc{G}_1)^{n_k}$ we have 
\eq{
\JG \vec{z}^k = \pmatsmall{i\gw_k- \mc{G}_1&I\\&i\gw_k-\mc{G}_1\\&&\ddots&I\\&&&i\gw_k-\mc{G}_1} \pmat{z\\\vdots\\0}
=  \pmat{(i\gw_k-\mc{G}_1)z \\\vdots\\ 0}
=  \pmat{\mc{G}_2y \\ \vdots\\ 0}
}
and 
\eq{
\MoveEqLeft\mc{G}_2 \left( \mathbb{P}(i\gw_k)K \vec{z}^k + C\mathbb{R}(i\gw_k,A)\tilde{E}\Phi_k + \tilde{F}\Phi_k \right)\\
&= \mc{G}_2\bigl[ (P(i\gw_k)Kz ,0,\ldots,0)^T + (y-P(i\gw_k)Kz ,0,\ldots,0)^T \bigr]
= \left( \mc{G}_2y , 0, \ldots, 0\right)^T .
}
Therefore $\vec{z}^k=(z,0,\ldots,0)^T$ is a solution of the equation~\eqref{eq:RORPcharunique}. Using the fact that 
\eq{
\tilde{B}_eP_k 
= \pmat{\tilde{E}P_k\\ \mc{G}_2 \tilde{F}P_k}
= \pmat{\tilde{E}\\ \mc{G}_2 \tilde{F}}
=\tilde{B}_e,
}
 we have from Lemma~\ref{lem:RORPcharunique} that the Sylvester equation $\Sigma S = \tilde{A}_e \Sigma + \tilde{B}_e$ has a solution $\Sigma\in \Lin(W,X_e)$. This concludes that $(A,B,C,D,\tilde{E},\tilde{F})\in \Ops$.
 
 Since $\vec{z}^k$ is a solution of~\eqref{eq:IMSfreq1}, we have from Corollary~\ref{cor:IMSindfreq} that it also satisfies
\eq{
0 &= \mathbb{P}(i\gw_k)K \vec{z}^k + C\mathbb{R}(i\gw_k,A)\tilde{E}\Phi_k + \tilde{F}\Phi_k \\
&= (P(i\gw_k)Kz ,0,\ldots,0)^T + (y-P(i\gw_k)Kz ,0,\ldots,0)^T 
= (y,0,\ldots,0)^T\in Y^{n_k},
}
which implies $y=0$. This further shows that $w=\mc{G}_2 y=0$. Since $w\in \ran(i\gw_k-\mc{G}_1)\cap \ker(\mc{G}_2) $ and $k\in\Z$ were arbitrary, we have that~\eqref{eq:Gconds1} is satisfied.
\end{proof}

\begin{lemma}
  \label{lem:RobtoGc2}
If the controller\/ $(\mc{G}_1,\mc{G}_2,K)$ solves the robust output regulation problem, then\/ \eqref{eq:Gconds2} is satisfied.
\end{lemma}

\begin{proof} 
  Let $y\in\ker(\mc{G}_2)$ and let $\phi\in W_{\ga+1}$ be such that $\norm{\phi}=1$.
Leave the operators $(A,B,C,D)$ unperturbed, and choose $\tilde{E}=0\in \Lin(W,X)$ and $\tilde{F} = \iprod{\cdot}{\phi} y\in \Lin(W,Y)$. 
 The operators $(A,B,C,D,\tilde{E},\tilde{F})$ satisfy the parts (a) and (b) of Assumption~\ref{ass:pertclass}.

 If we choose $\Sigma = 0 \in \Lin(W_\ga,X_e)$, then $\Sigma (W_{\ga+1})= \set{0}\subset \Dom(\tilde{A}_e)$ and for all $v\in W_{\ga+1}$ we have $\Sigma S v = 0$ and
 \eq{
 \tilde{A}_e \Sigma v + \tilde{B}_e v = \pmat{\tilde{E}v\\\mc{G}_2 \tilde{F}v}
 = \pmat{0\\\iprod{v}{\phi}\mc{G}_2 y} = 0.
 }
 This shows that $\Sigma =0$ is a solution of  $\Sigma S = \tilde{A}_e + \tilde{B}_e$, and thus $(A,B,C,D,\tilde{E},\tilde{F})\in \Ops$.

 Since the controller solves the robust output regulation problem, we have from Lemma~\ref{lem:RORP} that $\tilde{C}_e \Sigma + \tilde{D}_e = 0$ on $W_{\ga+1}$. In particular, using $\norm{\phi}=1$ gives
\eq{
0 = \tilde{C}_e \Sigma \phi + \tilde{D}_e \phi
= \tilde{F}\phi
= \iprod{\phi}{\phi} y 
= y.
}
Since $y\in \ker(\mc{G}_2)$ was arbitrary, this concludes the proof.
\end{proof}

\begin{lemma}
  \label{lem:RobtoGc3}
If $Z=\ran(i\gw_k-\mc{G}_1)+\ran(\mc{G}_2)$ for all $k\in\Z$, and if the controller\/ $(\mc{G}_1,\mc{G}_2,K)$ solves the robust output regulation problem, then\/ \eqref{eq:Gconds3} is satisfied.
\end{lemma}

\begin{proof} 
  Let $k\in\Z$ and $z\in\ker(i\gw_k-\mc{G}_1)^{n_k-1}$. Since $Z=\ran(i\gw_k-\mc{G}_1)+\ran(\mc{G}_2)$, there exist $z_1\in\Dom(\mc{G}_1)$ and $y\in Y$ such 
that
\eq{
z=(i\gw_k-\mc{G}_1)z_1+\mc{G}_2y.
}
To prove the claim it is now sufficient to show that $y=0$.
Leave the operators $(A,B,C,D)$ unperturbed, and choose $\tilde{E}=0\in \Lin(W,X)$.
Choose $\vec{z}^k\in \Dom(\mc{G}_1)^{n_k}$ in such a way that
\eq{
\vec{z}^k = \left( (-1)^{n_k-1}z_1, (-1)^{n_k-2}z, (-1)^{n_k-3} (i\gw_k-\mc{G}_1)z, \ldots ,  (i\gw_k-\mc{G}_1)^{n_k-2}z  \right)^T,
}
i.e. $\vec{z}^k = \left( (\vec{z}^k)_{n_k},\ldots,(\vec{z}^k)_1 \right)^T$ 
with components $(\vec{z}^k)_{n_k}= (-1)^{n_k-1}z_1$ and
\eq{
(\vec{z}^k)_{l} = (-1)^{l-1} (i\gw_k - \mc{G}_1)^{n_k-1-l} z
}
for $l=\List{n_k-1}$.  
Choose the operator $\tilde{F} \in \Lin(W,Y)$
in such a way that 
\eq{
\tilde{F} = \left( \sum_{l=1}^{n_k-1} -\iprod{\cdot}{\phi_k^l} (\mathbb{P}(i\gw_k)K\vec{z}^k)_l  \right) + \iprod{\cdot}{\phi_k^{n_k}} ( (-1)^{n_k} y - (\mathbb{P}(i\gw_k)K\vec{z}^k)_{n_k})
}
where $(\mathbb{P}(i\gw_k)K\vec{z}^k)_{l}$ denotes a component of the $n_k$-dimensional vector $\mathbb{P}(i\gw_k)K\vec{z}^k = \left( (\mathbb{P}(i\gw_k)K\vec{z}^k)_{n_k},\ldots,(\mathbb{P}(i\gw_k)K\vec{z}^k)_1 \right)^T$. 
The operators $(A,B,C,D,\tilde{E},\tilde{F})$ satisfy parts (a) and (b) of Assumption~\ref{ass:pertclass}.

We have
\ieq{
\tilde{F}\Phi_k = -\mathbb{P}(i\gw_k)K\vec{z}^k + \left( (-1)^{n_k}y,0,\ldots,0 \right)^T.
}
Now
\eq{
\MoveEqLeft[1]\mc{G}_2 \left( \mathbb{P}(i\gw_k)K\vec{z}^k + C \mathbb{R}(i\gw_k,A)\tilde{E}\Phi_k + \tilde{F}\Phi_k \right)\\
&= \mc{G}_2 \left( \mathbb{P}(i\gw_k)K\vec{z}^k -\mathbb{P}(i\gw_k)K\vec{z}^k +  \left( (-1)^{n_k}y,0,\ldots,0 \right)^T  \right)
= \left( (-1)^{n_k}\mc{G}_2 y,0,\ldots,0 \right)^T.
}
On the other hand,
\eq{
\MoveEqLeft[1] \JG \vec{z}^k
 = 
\pmatsmall{i\gw_k- \mc{G}_1&I\\&i\gw_k-\mc{G}_1\\&&\ddots&I\\&&&i\gw_k-\mc{G}_1}
\pmat{(-1)^{n_k-1}z_1\\(-1)^{n_k-2}z\\\vdots\\ (i\gw_k-\mc{G}_1)^{n_k-2}z}\\[1ex]
&=\pmat{(-1)^{n_k-1}( (i\gw_k-\mc{G}_1) z_1-z )\\(-1)^{n_k-2}( (i\gw_k-\mc{G}_1)z-(i\gw_k-\mc{G}_1)z)\\\vdots\\(-1)( (i\gw_k-\mc{G}_1)^{n_k-2}z - (i\gw_k-\mc{G}_1)^{n_k-2}z ) \\ (i\gw_k-\mc{G}_1)^{n_k-1}z}
=\pmat{(-1)^{n_k}  \mc{G}_2y \\0\\\vdots \\ 0}
}
where we have used $(i\gw_k- \mc{G}_1)^{n_k-1}z=0$ and $(i\gw_k-\mc{G}_1)z_1-z=-\mc{G}_2y$.
Therefore $\vec{z}^k$ is a solution of the equation~\eqref{eq:RORPcharunique}. Using the fact that $\tilde{B}_eP_k =\tilde{B}_e$, we have from Lemma~\ref{lem:RORPcharunique} that the Sylvester equation $\Sigma S = \tilde{A}_e \Sigma + \tilde{B}_e$ has a solution $\Sigma\in \Lin(W,X_e)$. This concludes that $(A,B,C,D,\tilde{E},\tilde{F})\in \Ops$.

Since $\vec{z}^k$ is a solution of~\eqref{eq:IMSfreq1}, we have from Corollary~\ref{cor:IMSindfreq} that it also satisfies
\eq{
0 &= \mathbb{P}(i\gw_k)K \vec{z}^k + C\mathbb{R}(i\gw_k,A)\tilde{E}\Phi_k + \tilde{F}\Phi_k 
= \tilde{F}\Phi_k \\
&= \mathbb{P}(i\gw_k)K \vec{z}^k - \mathbb{P}(i\gw_k)K \vec{z}^k + ( (-1)^{n_k} y,0,\ldots,0)^T
=  ( (-1)^{n_k} y,0,\ldots,0)^T\in Y^{n_k},
}
which implies $y=0$, and we therefore have $z = (i\gw_k-\mc{G}_1)z_1 \in \ran(i\gw_k-\mc{G}_1)$. Since $z\in \ker(i\gw_k- \mc{G}_1)^{n_k-1}$ was arbitrary, this concludes the proof.
\end{proof}

Finally, Lemma~\ref{lem:GctoRob} proves that the \Gconds\ are sufficient for the robustness of the controller.

\begin{lemma}\label{lem:GctoRob}
Assume that the controller\/ $(\mc{G}_1,\mc{G}_2,K)$ satisfies the \Gconds, the closed-loop system is strongly stable, and the Sylvester equation $\Sigma S = A_e\Sigma + B_e$ has a solution $\Sigma \in \Lin(W_\ga,X_e)$.
Then the controller solves the robust output regulation problem on $W_\ga$.
\end{lemma}

\begin{proof} 
  In this proof we will show that for all perturbations in $\Ops$ and for all $k\in\Z$ the unique solution $\vec{z}^k$ of~\eqref{eq:IMSfreq1} satisfies~\eqref{eq:IMSfreq2}. Since this will in particular be true for the operators $(A,B,C,D,E,F)$ of the unperturbed plant, the results in Section~\ref{sec:RedIMs} conclude that the solution $\Sigma$ of the Sylvester equation $\Sigma S = A_e\Sigma + B_e$ satisfies $C_e \Sigma+D_e=0$ on $W_{\ga+1}$. Therefore, by Theorem~\ref{thm:ORP} the controller solves the robust output regulation problem.
  Moreover, since the unique solutions of~\eqref{eq:IMSfreq1} satisfy~\eqref{eq:IMSfreq2} also for all other perturbations in $\Ops$, Corollary~\ref{cor:IMSindfreq} will conclude that the controller is robust with respect to all perturbations in $\Ops$, and thus solves the robust output regulation problem on $W_\ga$.

  Let $(\tilde{A},\tilde{B},\tilde{C},\tilde{D},\tilde{E},\tilde{F})\in \Ops$. 
Fix $k\in\Z$ and let $\vec{z}^k\in \Dom(\mc{G}_1)^{n_k}$ be the unique solution of~\eqref{eq:IMSfreq1}, i.e.,
  \eqn{
  \label{eq:GctoRob}
  \pmatsmall{i\gw_k- \mc{G}_1&I\\&i\gw_k-\mc{G}_1\\&&\ddots&I\\&&&i\gw_k-\mc{G}_1}\vec{z}^k
  = \mc{G}_2
  \left( \mathbb{P}(i\gw_k)K \vec{z}^k + C\mathbb{R}(i\gw_k,A)\tilde{E}\Phi_k + \tilde{F}\Phi_k \right) .
  \hspace{-2ex}
  }
  For brevity denote $\vec{y} = (y_{n_k},\ldots,y_1)^T=\mathbb{P}(i\gw_k)K \vec{z}^k + C\mathbb{R}(i\gw_k,A)\tilde{E}\Phi_k + \tilde{F}\Phi_k $ and $\vec{z}^k = (z_{n_k}^k,\ldots,z_1^k)$. The bottom line of equation~\eqref{eq:GctoRob} is 
  \ieq{
  (i\gw_k - \mc{G}_1)z_1^k = \mc{G}_2 y_1.
  } 
Now conditions~\eqref{eq:Gconds1} and~\eqref{eq:Gconds2} imply that $(i\gw_k-\mc{G}_1)z_1^k=0$ and $y_1=0$.

If $n_k\geq 2$,
the condition~\eqref{eq:Gconds3} implies
\eq{
z_1^k\in\ker(i\gw_k-\mc{G}_1) \subset\ker(i\gw_k-\mc{G}_1)^{n_k-1} \subset\ran(i\gw_k-\mc{G}_1).
}
The second last line of~\eqref{eq:GctoRob} is
\ieq{
(i\gw_k - \mc{G}_1)z_2^k + z_1^k = \mc{G}_2 y_2.
} 
Since $z_1^k\in\ran(i\gw_k-\mc{G}_1)$, conditions~\eqref{eq:Gconds1} and~\eqref{eq:Gconds2} imply  
\ieq{
(i\gw_k-\mc{G}_1)z_2^k + z_1^k = 0
}
and $y_2=0$.
In particular this also shows that $z_2^k\in \ker( (i\gw_k-\mc{G}_1)^2)$ since $(i\gw_k-\mc{G}_1)^2 z_2^k = - (i\gw_k - \mc{G}_1)z_1^k = 0$.

By repeating the previous step as many times as necessary we can show that $y_l=0$ and
\eq{
z_l^k\in\ker(i\gw_k-\mc{G}_1)^l \subset\ker(i\gw_k-\mc{G}_1)^{n_k-1} \subset\ran(i\gw_k-\mc{G}_1)
}
for all $l\in\List{n_k-1}$. Finally, the top line of the equation~\eqref{eq:GctoRob} is equal to
\ieq{
(i\gw_k - \mc{G}_1)z_{n_k}^k + z_{n_k-1}^k = \mc{G}_2 y_{n_k}.
} 
Since $z_{n_k-1}^k\in\ran(i\gw_k-\mc{G}_1)$, conditions~\eqref{eq:Gconds1} and~\eqref{eq:Gconds2} imply  
\ieq{
(i\gw_k-\mc{G}_1)z_{n_k}^k + z_{n_k-1}^k = 0
}
and $y_{n_k}=0$.
We have now concluded that
\eq{
0=\vec{y} =\mathbb{P}(i\gw_k)K \vec{z}^k + C\mathbb{R}(i\gw_k,A)\tilde{E}\Phi_k + \tilde{F}\Phi_k ,
}
and thus we have shown that the unique solution $\vec{z}^k$ of~\eqref{eq:IMSfreq1} satisfies equations~\eqref{eq:IMSfreq2}. 

As stated in the beginning of the proof, the fact that $(\tilde{A},\tilde{B},\tilde{C},\tilde{D},\tilde{E},\tilde{F})\in \Ops$ and $k\in\Z$ were arbitrary allows us to conclude that the controller solves the robust output regulation problem on $W_\ga$.
\end{proof}

\section{Regular Linear Systems}
\label{sec:regsys}

In this section we show that Assumption~\ref{ass:CLsysass} is in particular satisfied if the plant and the controller are regular linear systems~\cite{Wei94,CurWei97,Sta05book}.
The operator $B$ is said to be an \keyterm{admissible input operator} (with respect to the semigroup $T(t)$ generated by $A$) if for some $\tau>0$  (and consequently for all $\tau>0$) and $u\in \Lp[2](0,\tau;U)$~\cite[Sec. 4.2]{TucWei09}
\eq{
\int_0^\tau T_{-1}(\tau-s)Bu(s)ds\in X.
}
Moreover, the operator $C$ is called an \keyterm{admissible output operator} if for one/all $\tau>0$ there exists $c_\tau>0$ such that
\eq{
\int_0^\tau \norm{CT(s)x }^2 ds \leq c_\tau\norm{x}^2 \qquad \forall x\in \Dom(A).
} 
The admissibility of the output operator $K\in \Lin(Z_1,U)$ of the controller with respect to the semigroup generated by $\mc{G}_1$ is defined analogously.
For admissible operators $C$ and $K$, we can define their \keyterm{$\Lambda$-extensions} by~\cite{Wei94,CurWei97}
\eq{
C_\Lambda x = \lim_{\gl\to\infty} \gl C(\gl -A)\inv x 
\qquad \mbox{and} \qquad
K_\Lambda z = \lim_{\gl\to\infty} \gl K(\gl -\mc{G}_1)\inv z,
}
with domains $\Dom(C_\Lambda)$ and $\Dom(K_{\Lambda})$ consisting of those elements $x\in X$ and $z\in Z$, respectively, for which the limits exist.
In the system equations (as well as elsewhere in the paper), the admissible operators $C$ and $K$ can be replaced without loss of generality with their $\Lambda$-extensions $C_\Lambda$ and $K_\Lambda$.

The plant $(A,B,C,D)$ with admissible input and output operators is said to be a \keyterm{regular linear system} if $\ran(R(\gl,A_{-1})B) \subset \Dom(C_\Lambda )$ for one/all $\gl\in \rho(A)$ (which is one of our standing assumptions made in Section~\ref{sec:plantexo}), and if $\gl\mapsto\norm{P(\gl)}$ is uniformly bounded on some right half-plane $\C_{\gb}^+$~\cite[Prop. 2.1]{CurWei97}.

If the operator $K$ is an admissible output operator with respect to the semigroup generated by $\mc{G}_1$ and the operator $\mc{G}_2$ is bounded, then also the controller $(\mc{G}_1,\mc{G}_2,K)$ is a regular linear system (due to~\cite[Thm. 4.3.7]{TucWei09}).

\begin{theorem}
  \label{thm:regsys}
  If both the plant $(A,B,C,D)$ and the controller $(\mc{G}_1,\mc{G}_2,K)$ with $\mc{G}_2\in \Lin(Y,Z)$ are regular linear systems, then Assumption~\textup{\ref{ass:CLsysass}} is satisfied. 
\end{theorem}

\begin{proof}
  The plant (without the disturbance signal $w(t)$)
  and the controller can be written together as a composite open loop system
  \eq{
  \ddb{t} \pmat{x\\z} &= \pmat{A&0\\0&\mc{G}_1} \pmat{x\\z} + \pmat{0&B\\\mc{G}_2&0} \pmat{e\\u}\\
  \pmat{y\\u}&= \pmat{C_\Lambda&0\\0&K_\Lambda} \pmat{x\\z} + \pmat{0&D\\0&0} \pmat{e\\u}.
  }
  Denote $\hat{x}= (x,z)^T$, $\hat{y}=(y,u)^T$, $\hat{u}=(e,u)^T$,
  \eq{
  \hat{A} = \pmat{A&0\\0&\mc{G}_1}, \quad \hat{B}= \pmat{0&B\\\mc{G}_2&0}, 
  \quad \hat{C}_\Lambda= \pmat{C_\Lambda&0\\0&K_\Lambda} , \quad \hat{D}= \pmat{0&D\\0&0} .
  }
We will show that $(\hat{A},\hat{B},\hat{C},\hat{D})$ is regular a regular linear system on $X_e = X\times Z$. The operator $\hat{A}$ generates a strongly continuous semigroup on $X_e$, and it is immediate that $\hat{B}$ and $\hat{C}$ are admissible with respect to $\hat{A}$. We have $\ran(R(\gl,\hat{A}_{-1})\hat{B})\subset \Dom(\hat{C})$ for all $\gl\in \rho(\hat{A}) = \rho(A)\cap \rho(\mc{G}_1)$, and 
the transfer function of $(\hat{A},\hat{B},\hat{C},\hat{D})$ is given by
  \eq{
  \hat{P}(\gl) &= \hat{C}_\Lambda R(\gl,\hat{A}_{-1})\hat{B} + \hat{D}\\
  &= \pmat{C_\Lambda&0\\0&K_\Lambda} \pmat{R(\gl,A_{-1})&0\\0&R(\gl,\mc{G}_{1,-1})}  \pmat{0&B\\\mc{G}_2&0}+ \pmat{0&D\\0&0} \\
  &= \pmat{0&C_\Lambda R(\gl,A_{-1})B + D \\K_\Lambda R(\gl,\mc{G}_1) \mc{G}_2& 0 }
 =\pmat{0&P(\gl)\\ P_{\mc{G}}(\gl)& 0 } .
  }
  Since $(A,B,C,D)$ and $(\mc{G}_1,\mc{G}_2,K)$ are regular linear systems, the mapping $\gl\mapsto\norm{\hat{P}(\gl)}$ is bounded on some half-plane $\C_{\hat{\gb}}^+$. This concludes that $(\hat{A},\hat{B},\hat{C},\hat{D})$ is a regular linear system~\cite[Prop. 2.1]{CurWei97}.

  We will show that the the operators $A_e$ and $C_e$ are the system operator and the output operator, respectively, of a linear system that is obtained from the open loop system $(\hat{A},\hat{B},\hat{C},\hat{D})$ by applying a static output feedback $\hat{u}=\hat{K} \hat{y} + \tilde{u}$ with $\hat{K} = \bigl( {I \atop 0} ~ {0\atop I} \bigr)$. Once we show that $\hat{K}$ is an admissible feedback operator for $(\hat{A},\hat{B},\hat{C},\hat{D})$, the theory in~\cite{Wei94,CurWei97} concludes that the closed-loop system resulting from the static output feedback is regular as well. This will in particular imply that $A_e$ generates a strongly continuous semigroup on $X_e$ and that $C_e$ is relatively bounded with respect to $A_e$.

  We begin by showing that $\hat{K}$ is an admissible feedback operator for $(\hat{A},\hat{B},\hat{C},\hat{D})$. To do this, we need to show that on some right half-plane of $\C$ the inverses $(I-\hat{P}(\gl)\hat{K})\inv$ exist and are uniformly bounded. This is achieved if we can find $\gb'\in \R$ and  $0<\gg<1$ so that $\norm{\hat{P}(\gl)\hat{K}}\leq \gg<1$ for all $\gl$ in the half-plane $\C_{\gb'}^+$. 

  Since $\mc{G}_2$ is bounded and $K$ is admissible, by~\cite[Thm. 4.3.7]{TucWei09} there exist $\gw\in\R$ and $\tilde{M}>0$ such that for every $\gl\in\C$ with $\re\gl>\gw$ we have
  \ieq{
  \norm{P_{\mc{G}}(\gl)}=\norm{KR(\gl,\mc{G}_1)\mc{G}_2}
  \leq \tilde{M}\norm{\mc{G}_2}/\sqrt{\re\gl - \gw},
  }
  and thus $\norm{P_{\mc{G}}(\gl)}\rightarrow 0$ as $\re\gl\rightarrow \infty$.
Since $(A,B,C,D)$ is regular, $P(\cdot)$ is uniformly bounded on some right half-plane of $\C$. We can therefore choose $\gb'>\gw$ in such a way that $P(\cdot)$ and $P_{\mc{G}}(\cdot)$ are uniformly bounded on $\C_{\gb'}^+$ and $\norm{P_{\mc{G}}(\gl)}\norm{P(\gl)}\leq \gg<1$ for every $\gl\in\C_{\gb'}^+$. We then have that $(I-P_{\mc{G}}(\gl)P(\gl))\inv$ exists and $\norm{(I-P_{\mc{G}}(\gl)P(\gl))\inv}\leq 1/(1-\gg)$ for all $\gl\in \C_{\gb'}^+$.
Furthermore, for every $\gl\in \C_{\gb'}^+$ we have
\eq{
(I -  \hat{P}(\gl) \hat{K} )\inv
  = \pmat{I&P(\gl)\\0&I}\pmat{I&0\\0&(I-P_{\mc{G}}(\gl)P(\gl))\inv} \pmat{I&0\\P_{\mc{G}}(\gl)&I},
}
and
\eq{
\norm{(I -  \hat{P}(\gl) \hat{K} )\inv}
&\leq (2+\norm{P(\gl)})\max\set{1,\norm{(I-P_{\mc{G}}(\gl)P(\gl))\inv}} (2+\norm{P_{\mc{G}}(\gl)} )\\
&\leq \frac{1}{1-\gg}(2+\norm{P(\gl)}) (2+\norm{P_{\mc{G}}(\gl)} ),
}
which is uniformly bounded on $\C_{\gb'}^+$. This concludes that $\hat{K}$ is an admissible feedback operator for $(\hat{A},\hat{B},\hat{C},\hat{D})$.

By~\cite{Wei94}, \cite[Sec. II]{CurWei97} the closed-loop system $(\hat{A}^K,\hat{B}^K,\hat{C}^K,\hat{D}^K)$ obtained with output feedback $\hat{u}=\hat{K} \hat{y}+ \tilde{u}$ is a regular linear system. The operators $\hat{A}^K$ and $\hat{C}_\Lambda^K$ can be expressed using the operator
  \eq{
  (I - \hat{D} \hat{K})\inv 
  = \pmat{I&-D\\0&I}\inv
  = \pmat{I&D\\0&I}.
  }
  The generator $\hat{A}^K$ is given by a formula~\cite[Sec. II]{CurWei97}
  \eq{
  \hat{A}^K \hat{x}
  &= (\hat{A} + \hat{B}\hat{K}(I-\hat{D}\hat{K})\inv \hat{C}_\Lambda)\hat{x}\\
  &= \pmat{A&0\\0&\mc{G}_1}\pmat{x\\z}
  + \pmat{0&B\\\mc{G}_2&0}\pmat{I&D\\0&I}\pmat{C_\Lambda&0\\0&K_\Lambda}\pmat{x\\z}\\
  &= \pmat{A&BK_\Lambda\\\mc{G}_2 C_\Lambda &\mc{G}_1+ \mc{G}_2DK_\Lambda}\pmat{x\\z}
  }
  with domain
  \eq{
  \Dom(\hat{A}^K)
  &= \Setm{\pmatsmall{x\\z}\in \Dom(C_\Lambda)\times \Dom(K_\Lambda)}{ (\hat{A} + \hat{B}\hat{K}(I-\hat{D}\hat{K})\inv \hat{C}_\Lambda)\pmatsmall{x\\z} \in X\times Z}\\
  &= \Setm{\pmatsmall{x\\z}\in \Dom(C_\Lambda)\times \Dom(\mc{G}_1)}{ Ax + BK_\Lambda z \in X}.
  }
  This shows that $\hat{A}^K$ coincides with $A_e$ in Section~\ref{sec:plantexo}. Moreover,
  \eq{
  \hat{C}^K \hat{x}
  =  (I-\hat{D}\hat{K})\inv \hat{C}_\Lambda\hat{x}
  =\pmat{I&D\\0&I}\pmat{C_\Lambda&0\\0&K_\Lambda}\pmat{x\\z}
  = \pmat{C_\Lambda & DK_\Lambda\\ 0 & K_\Lambda}\pmat{x\\z}
  }
  with domain $\Dom(\hat{C}^K) = \Dom(C_\Lambda)\times \Dom(K_\Lambda)$. Because of this, the first lines of $\hat{C}^K$ coincides with the operator $C_e$ in Section~\ref{sec:plantexo}.
  Because $(\hat{A}^K,\hat{B}^K,\hat{C}^K,\hat{D}^K)$ is a regular linear system, the operator $A_e$ generates a strongly continuous semigroup and $C_e$ is an admissible observation operator (with respect to the semigroup $T_e(t)$), and relatively bounded with respect to $A_e$.
\end{proof}

\begin{remark}
  \label{rem:regAedom}
  As in \textup{\cite[Sec. II]{CurWei97}}, the domain of $\hat{A}^K$ can also be expressed in the form
  \eq{
  \Dom(\hat{A}^K)
  &= \Setm{\pmatsmall{x\\z}\in \hat{X}^1}{ (\hat{A} + \hat{B}\hat{K}(I-\hat{D}\hat{K})\inv \hat{C}_\Lambda)\pmatsmall{x\\z} \in X\times Z}
  }
  where $\hat{X}^1 = \Dom(\hat{A}) + \ran( R(\mu,\hat{A}_{-1})\hat{B})$ for some $\mu\in \rho(\hat{A})$.
This together with a straightforward computation shows that
\eq{
\Dom(A_e) = \Setm{\pmatsmall{x\\z}\in X^1\times \Dom(\mc{G}_1)}{A_{-1}x+BKz\in X}
}
where $X^1 = \Dom(A) + \ran(R(\mu,A_{-1})B)$ for some $\mu\in \rho(A)$.
\end{remark}

\section{Robust Output Tracking for a Heat Equation}
\label{sec:heatex}

\newcommand{\z}{\xi}

In this section we consider robust output tracking for a stable one-dimensional heat equation with Neumann boundary control and point measurements. The system is given by
\eq{
\pd{x}{t}(\z,t) &= \pd[2]{x}{\z}(\z,t) - x(\z,t) \\
-\pd{x}{\z}(0,t)&=u_1(t), \quad
\pd{x}{\z}(1,t)=u_2(t)
}
with initial state $x(\z,0) = x_0(\z)$.
The temperature of the system is measured at two points
\eq{
y(t)=\pmat{x(1/\sqrt{8},t)\\x(1/ \sqrt{2},t)}.
} 
The plant can be written in the form~\eqref{eq:plantintro} if we choose $X = \Lp[2](0,1)$, $U=\C^2$, $Y=\C^2$, and
\eq{
(Ax)(\z) &= x''(\z) - x(\z),\\
 \Dom(A) &= \Setm{x\in X}{x,x' ~\mbox{abs. cont.} ~ x''\in \Lp[2](0,1), ~ x'(0)=x'(1)=0}.
}
The operator $A$ has a spectral representation~\cite[Ch. 2]{CurZwa95book}
\eq{
Ax &= \sum_{k=0}^\infty \gl_k \iprod{x}{\varphi_k}_{\Lp[2]}\varphi_k(\cdot)\\
x\in \Dom(A) &= \Bigl\{x\in X \; \Bigm| \;\sum_{k=0}^\infty \abs{\gl_k}^2 \abs{\iprod{x}{\varphi_k}_{\Lp[2]}}^2 <\infty\Bigr\},
}
where $\gl_k=-k^2\pi^2-1$, $\varphi_0(\z)\equiv 1$, and $\varphi_k(\z)= \sqrt{2} \cos(\pi k \z)$ for $k\in \N$. Thus the spectrum of $A$ satisfies $\gs(A)=\gs_p(A) = \set{\gl_k}_{k=0}^\infty$. 
The operator $A$ is boundedly invertible and generates an exponentially stable analytic semigroup on $X$. Since $\set{\varphi_k}$ is an orthonormal basis of $X$, we have $\norm{R(\gl,A)}= \min_k\abs{\gl-\gl_k}^{-1}$ for all $\gl\in \rho(A)$.
The space $X_{-1}$ is given by
\eq{
X_{-1} &= \Bigl\{ ~ \sum_{k=0}^\infty \iprod{x}{\varphi_k}\varphi_k \; \Bigm| \; \sum_{k=0}^\infty \frac{1}{\abs{\gl_k}^2} \abs{\iprod{x}{\varphi_k}}^2 <\infty \Bigr\}.
}
The operator $-A$ is positive and sectorial, and its fractional powers have representations
\eq{
(-A)^\gb x &= \sum_{k=0}^\infty (-\gl_k)^\gb \iprod{x}{\varphi_k}\varphi_k(\cdot)\\
x\in \Dom( (-A)^\gb) &= \Bigl\{\; \sum_{k=0}^\infty \iprod{x}{\varphi_k}\varphi_k \; \Bigm| \;\sum_{k=0}^\infty \abs{\gl_k}^{2\gb} \abs{\iprod{x}{\varphi_k}}^2 <\infty\Bigr\} 
}
for all $\gb\in\R$.

The boundary control can be written formally as $B \bigl( {u_1\atop u_2} \bigr) = b_1u_1+b_2u_2$ with $b_1(\z) = \gd(\z)$, and $b_2(\z) = \gd(\z-1)$ (where $\gd(\cdot)$ is the Dirac delta function). We have $b_1,b_2\in X_{-1}$.  
Similarly, the observation operator can be written as 
$C x = \bigl(\iprod{x}{c_1},\iprod{x}{c_2}\bigr)^T$
with $c_1(\z)=\gd(\z-1/\sqrt{8})$, $c_2(\z)=\gd(\z-1/\sqrt{2})$ and domain $\Dom(C)= \Setm{x\in X}{x(\cdot) ~ \mbox{is cont.}}$.
For any $\z_0\in[0,1]$ we have
\eq{
\MoveEqLeft[2] \sum_{k=0}^\infty \abs{\gl_k}^{2\cdot (-1/2)} \abs{\iprod{\gd(\cdot-\z_0)}{\varphi_k}}^2
= \sum_{k=0}^\infty \frac{\abs{\varphi_k(\z_0)}^2}{\abs{\gl_k}}
\leq 1+ \frac{2}{\pi^2}\sum_{k=1}^\infty \frac{1}{k^2}
=  \frac{4}{3}.
}
This shows that $b_1,b_2,c_1,c_2\in \Dom( (-A_{-1})^{-1/2})$, and that $\norm{(-A_{-1})^{-1/2}b_j}\leq \sqrt{4/3}$ and $\norm{(-A_{-1})^{-1/2}c_j}\leq \sqrt{4/3}$ for $j=1,2$, which further imply
$B\in \Lin(U,X_{-1})$, $C\in \Lin(X_1,Y)$, $\ran(R(\gl,A_{-1})B)\subset \Dom( (-A)^{1/2})\subset \Dom(C)$ and $P(\gl) = CR(\gl,A_{-1})B \in \Lin(U,Y)$ for all $\gl\in \rho(A)$.

We have $\iprod{b_1}{\varphi_0} = \varphi_0(0)=1$, $\iprod{b_2}{\varphi_0} = \varphi_0(1)=1$, and
$\iprod{b_1}{\varphi_k} = \varphi_k(0)=\sqrt{2}$, $\iprod{b_2}{\varphi_k} = \varphi_k(1)=\sqrt{2}\cos(\pi k) = \sqrt{2}(-1)^k$ for $k\in \N$.
Likewise, $\iprod{c_1}{\varphi_0} = \varphi_0(1/\sqrt{8})=1$, $\iprod{c_2}{\varphi_0} = \varphi_0(1/\sqrt{2})=1$, and $\iprod{c_1}{\varphi_k} = \varphi_k(1/\sqrt{8}) = \sqrt{2}\cos(\pi k/\sqrt{8})$, and $\iprod{c_2}{\varphi_k} = \varphi_k(1/\sqrt{2}) = \sqrt{2} \cos(\pi k/\sqrt{2})$ for $k\in\N$.
The transfer function of the plant has a series representation
\eq{
P(\gl)u 
&= \sum_{k = 0}^\infty \frac{1}{\gl-\gl_k} \iprod{Bu}{\varphi_k} C\varphi_k
= \sum_{k = 0}^\infty \frac{1}{\gl-\gl_k} \pmat{\iprod{\varphi_k}{c_1}\\\iprod{\varphi_k}{c_2}} (\iprod{b_1}{\varphi_k}u_1 +  \iprod{b_2}{\varphi_k}u_2 )\\
&=\frac{1}{\gl+1}\pmat{1&1\\1&1}u +  2\sum_{k=1}^\infty \frac{1}{\gl-\gl_k} \pmat{\cos(\pi k/\sqrt{8})\\\cos(\pi k/\sqrt{2})} \pmat{1, ~ (-1)^k}u,
}
and $\norm{P(\gl)}$ can be estimated as
\eqn{
\label{eq:Pnormest}
\norm{P(\gl)}
&\leq \frac{1}{\abs{\gl+1}} \Norm{\pmat{1&1\\1&1}} + 2 \sum_{k=1}^\infty \frac{1}{\abs{\gl-\gl_k}} \sqrt{2} \sqrt{2}
\leq  4 \sum_{k=0}^\infty \frac{1}{\abs{\gl-\gl_k}}.
}

\subsection{Robust Tracking of Constant Reference Signals}

In the first part of this example we consider a one-dimensional exosystem.
We choose its parameters as $W=\C$, $S = 0\in \C$, $E=0\in X$, $F=-\frac{1}{5}(1,3)^T\in \C^2$. Then for the initial state $v_0\in \C $ the reference signal generated by the exosystem is 
\eq{
\yref(t) = -Fe^{St}v_0 
= \frac{1}{5}\pmat{1\\3} v_0.
}

Our aim is to solve the robust output regulation problem using a 2-dimensional controller with an internal model.
We choose the parameters of the controller on $Z=\C^2$ in such a way that  $\mc{G}_1 = 0\in \C^{2\times 2}$, $\mc{G}_2 = \frac{1}{5} I$, and $K = -P(0)\inv $.
We will show that with these choices the closed-loop system operator $A_e$ generates an exponentially stable analytic semigroup. To this end, let $\gd=0.025$ and consider a sector
\eq{
\Sigma_\gd = \Setm{\gl \in \C}{\arg (\gl+\gd)> \frac{3\pi}{4}}.
}
We will show that outside this sector, i.e., on $\C\setminus \Sigma_\gd$, the resolvent operator $R(\gl,A_e)$ exists and satisfies $\abs{\gl + \gd} \norm{R(\gl,A_e)}\leq M$ for some constant $M>0$.

Let $\gl\in \rho(A)$, $(x,z)^T\in X_e$ and $(x_1,z_1)^T \in \Dom(A_e)$. A direct computation (on $X_{-1}\times Z$) shows that
\eq{
(\gl - A_e) \pmat{x_1\\z_1} = \pmat{x\\z}
\qquad \Leftrightarrow \qquad
&\left\{
\begin{array}{l}
  (\gl - A_{-1})x_1 - BKz_1 = x\\
  -\mc{G}_2 Cx_1 + (\gl - \mc{G}_1)z_1 - \mc{G}_2 DK z_1  = z
\end{array}
\right.\\
 \Leftrightarrow \qquad
&\left\{
\begin{array}{l}
  x_1 = R(\gl,A_{-1})BKz_1 + R(\gl,A)x\\
  (\gl - \mc{G}_1 - \mc{G}_2P(\gl)K)z_1   = \mc{G}_2 CR(\gl,A)x+ z.
\end{array}
\right. 
}
This shows that if $S_A(\gl)=\gl - \mc{G}_1 - \mc{G}_2 P(\gl)K$ (the Schur complement of $\gl - A$ in $\gl - A_e$) is boundedly invertible, then $\gl\in \rho(A_e)$ and 
\eq{
 R(\gl,A_e) \pmat{x\\z} 
 &=  \pmat{x_1\\z_1} 
 = \pmat{R(\gl,A_{-1})BKz_1 + R(\gl,A)x\\S_A(\gl)\inv(\mc{G}_2 CR(\gl,A)x+  z)}\\
 &= \pmat{R(\gl,A_{-1})BKS_A(\gl)\inv(\mc{G}_2 CR(\gl,A)x+  z) + R(\gl,A)x\\S_A(\gl)\inv(\mc{G}_2 CR(\gl,A)x+  z)}.
}
For all $\gl\in \rho(A)$ we have
\eq{
\MoveEqLeft[1]\norm{R(\gl,A_{-1})B}
=\norm{(-A)R(\gl,A)(-A_{-1})\inv B} 
\leq\norm{(\gl-A-\gl)R(\gl,A)}\norm{(-A_{-1})\inv B} \\
&=\norm{I-\gl R(\gl,A)}\norm{(-A_{-1})\inv B} 
\leq(1+\abs{\gl} \norm{R(\gl,A)})\norm{(-A_{-1})\inv B} .
}
The analyticity of the semigroup generated by $A$ thus implies that $\norm{R(\gl,A_{-1})B}$ is uniformly bounded outside the sector $\Sigma_\gd$. Analogously we can see that the same is true for $\norm{CR(\gl,A)}$. Because of this, the behaviour of $R(\gl,A_e)$ on $\C\setminus \Sigma_\gd$ is characterized by the behavior of $S_A(\gl)$.
As above, we have that if $\gl\in \rho(A)$ then 
\eq{
\MoveEqLeft[1] \norm{P(\gl)}
=\norm{CR(\gl,A_{-1})B}
=\norm{C(-A)^{-1/2}(\gl-A-\gl)R(\gl,A)(-A_{-1})^{-1/2} B} \\
&\leq(1+\abs{\gl} \norm{R(\gl,A)})\norm{C(-A)^{-1/2}}\norm{(-A_{-1})^{-1/2} B} ,
}
and thus $\norm{P(\gl)}$ is uniformly bounded on $\C\setminus \Sigma_\gd$. 

We have
\eq{
S_A(\gl)=\gl - \mc{G}_1 - \mc{G}_2 P(\gl)K
= \gl + \frac{1}{5} P(\gl)P(0)\inv
= \gl (I+  \frac{1}{5\gl} P(\gl)P(0)\inv),
}
and using estimate~\eqref{eq:Pnormest} we can see that for $0<q<1$ there exists $r>0$ such that
\eq{
\bigl\|\frac{1}{5 \gl} P(\gl)P(0)\inv\bigr\|
\leq\frac{4}{5 \abs{\gl}} \norm{P(0)\inv} \sum_{k=0}^\infty \frac{1}{\abs{\gl-\gl_k}}
\leq q<1 
}
whenever $\gl\in \C\setminus \Sigma_\gd$ satisfies $\abs{\gl}\geq r$. Straightforward estimates
can now be used to show that we can choose $r=4$. We then have
\eq{
\MoveEqLeft\sup_{\substack{\gl\in\C\setminus \Sigma_\gd\\ \abs{\gl}\geq 4}}\; \abs{\gl+\gd}\norm{S_A(\gl)\inv}
= \sup_{\substack{\gl\in\C\setminus \Sigma_\gd\\ \abs{\gl}\geq 4}}\; \frac{\abs{\gl+\gd}}{\abs{\gl}} \Bigl\|(I-\frac{1}{5\gl}P(\gl)P(0)\inv)\inv\Bigr\|\\
&\leq \sup_{\substack{\gl\in\C\setminus \Sigma_\gd\\ \abs{\gl}\geq 4}}\; \frac{\abs{\gl+\gd}}{\abs{\gl}} \frac{1}{1-q}<\infty.
} 
Moreover,
using the series representation for $P(\gl)$, we can 
numerically verify that the values $\gl$ with $\abs{\gl}\leq 4$ for which $S_A(\gl)$ is not invertible belong to the sector $\Sigma_\gd$. 
Therefore, for $\gl\in \C\setminus \Sigma_\gd$ with $\abs{\gl}\leq 4$ the norms $\norm{S_A(\gl)\inv}$ are uniformly bounded. Together these estimates conclude that there exists $M_1>0$ such that $\abs{\gl+\gd}\norm{S_A(\gl)\inv}\leq M_1$ for $\gl\in \C\setminus\Sigma_\gd$.
The above properties of the closed-loop system finally conclude that $\gs(A_e)\subset \Sigma_\gd$ and there exists $M>0$ such that 
\ieq{
\norm{R(\gl,A_e)}\leq \frac{M}{\abs{\gl+\gd}} 
}
for all $ \gl\in \C\setminus \Sigma_\gd$, $\gl\neq -\gd$. Thus the closed-loop system is analytic and exponentially stable.

It remains to verify that the operator $C_e$ is $A_e$-bounded. For any $x_e=(x,z)^T \in \Dom(A_e)\subset \Dom(C)\times \C$ we have
\eq{
 \norm{C_e x_e}^2 
= \norm{Cx + DKz}^2
= \norm{Cx}^2
\leq  \norm{A_{-1}x + BKz}^2 + 5^2 \bigl\|\frac{1}{5} Cx\bigr\|^2 
\leq
25  \norm{A_e x_e}^2.
}
This concludes that $C_e$ is relatively bounded with respect to $A_e$.

Since the closed-loop system is analytic and exponentially stable, and since $S=0$ generates a bounded group on $W=\C$, the Sylvester equation $\Sigma S = A_e \Sigma + B_e$ has a solution $\Sigma\in \Lin(W,X_e)$~\cite[Cor. 8]{Vu91}.

Since $i\gw_0=0$ and $\dim \ker(i\gw_0-\mc{G}_1)=\dim\C^2=2$, the controller $(\mc{G}_1,\mc{G}_2,K)$ incorporates a p-copy internal model of the exosystem. Theorem~\ref{thm:IMP} thus concludes that the controller solves the robust output regulation problem. More precisely, the controller achieves asymptotic tracking of constant reference signals, and this property is robust with respect to any perturbations that preserve the closed-loop stability and the solvability of the Sylvester equation. In particular, these include sufficiently small bounded perturbations of the operators $(A,B,C,D)$ --- under which the closed-loop system remains analytic and exponentially stable --- as well as arbitrary perturbations of the operators $E$ and $F$ (which do not affect the closed-loop system operator).

The behaviour of the closed-loop system was simulated on the time interval $[0,30]$ using a truncated eigenfunction expansion for $A$ with $N=31$ eigenfunctions $\varphi_k(\cdot)$. The initial states of the plant and the controller were chosen as $x_0(\z)=\frac{1}{4}\z^3-\frac{3}{8}\z^2-\frac{1}{4}$ (which satisfies $x_0\in \Dom(A)$) and $z_0=0\in \C^2$. Together the initial states satisfy $x_{e0}=(x_0,z_0)^T\in \Dom(A_e)$.
Figure~\ref{fig:heatex-sim-sol} describes the behaviour of the state $x(\z,t)$ of the controlled system, and the output $y(t)$ of the controlled system is depicted in Figure~\ref{fig:heatex-sim-output}.

\begin{figure}[ht]
\begin{minipage}[t]{0.48\linewidth}
  \begin{center}
      \includegraphics[width=0.9\linewidth]{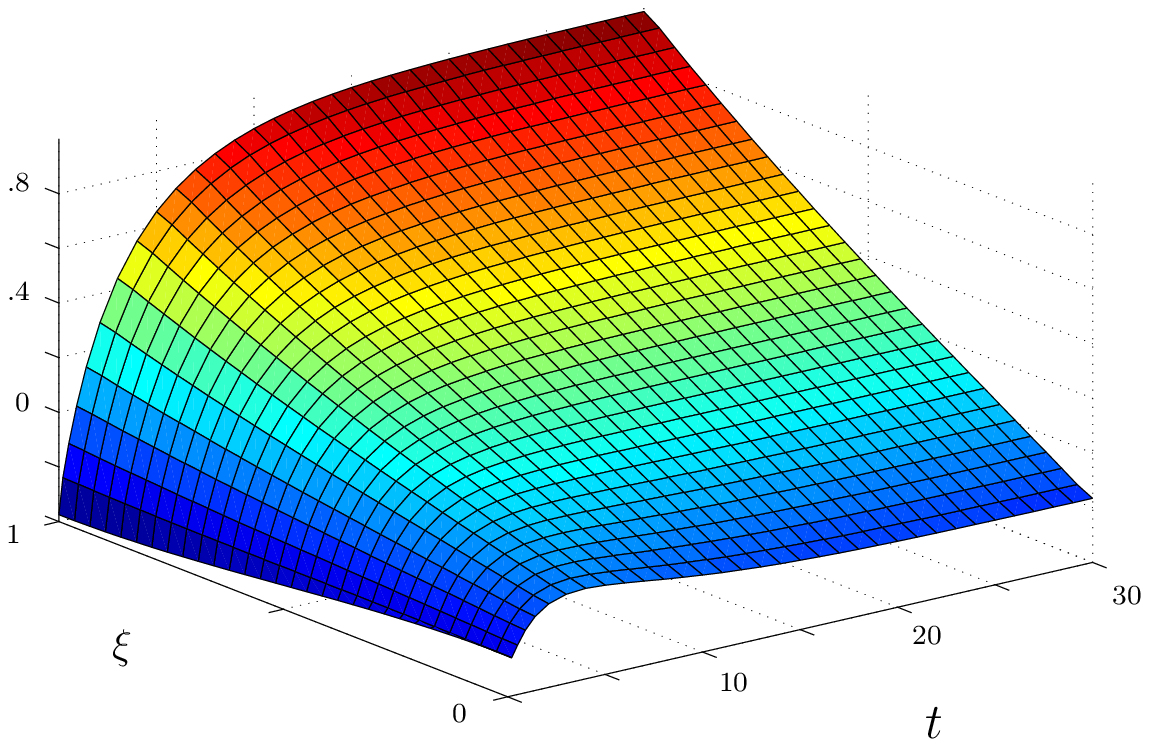}
  \end{center}
  \caption{State $x(\z,t)$ of the controlled system.}
  \label{fig:heatex-sim-sol}
\end{minipage}
\hfill
\begin{minipage}[t]{0.48\linewidth}
  \begin{center}
      \includegraphics[width=0.8\linewidth]{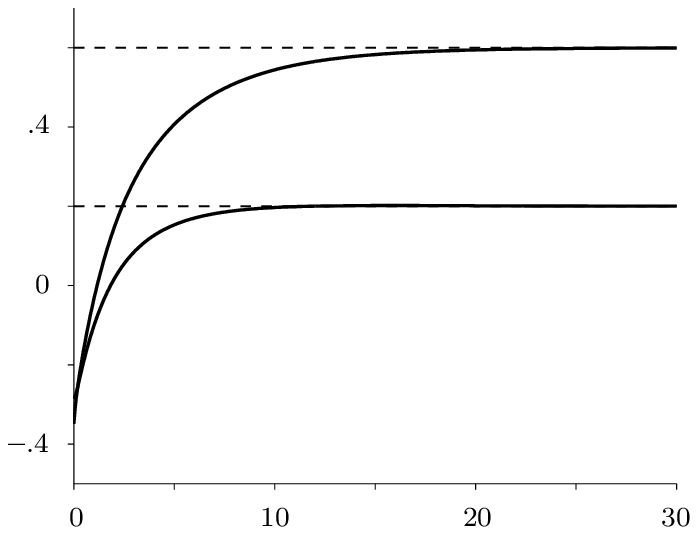}
  \end{center}
  \caption{Output $y(t)$ of the controlled system.}
  \label{fig:heatex-sim-output}
\end{minipage}
\end{figure}

\subsection{Robust Tracking of Continuous Periodic Signals}
\label{sec:infexoex}

We conclude the example by considering the tracking of continuous periodic signals. In particular our approach illustrates dividing the robust output regulation problem into two parts. In the first part we fix the structure of the controller in such a way that the controller incorporates an internal model of the exosystem. The second part of the problem consists of choosing the remaining parameters of the controller in such a way that the closed-loop system is strongly stable and the Sylvester equation $\Sigma S = A_e\Sigma + B_e$ has a solution. In this paper we have not considered techniques for stabilizing the closed-loop system, and therefore the question on how to choose these free parameters is left open.

We consider an infinite-dimensional exosystem on the space $W=\lp[2](\C)$. Choosing $\phi_k = e_k$, the natural basis of $\lp[2](\C)$, we define
\eq{
S = \sum_{k\in\Z} ik \iprod{\cdot}{\phi_k}\phi_k,
\qquad \Dom(S) = \Bigl\{v\in W ~\Bigm|~\sum_{k\in\Z} \, k^2 \abs{\iprod{v}{\phi_k}}^2<\infty\Bigr\}
}
and $F\in \Lin(W,\C)$ is chosen in such a way that $F\phi_0=1$ and $F\phi_k = 1/\abs{k}^{3/5}$ for all $k\neq 0$. 
For this exosystem the reference signals to be tracked are of the form
\eq{
\yref(t) &= FT_S(t) v_0 
= \sum_{k\in\Z} e^{ik t} \iprod{v_0}{\phi_k} F\phi_k 
= \iprod{v_0}{\phi_0} + \sum_{k\neq 0} e^{ik t} \frac{\iprod{v_0}{\phi_k}}{\abs{k}^{3/5}},
}
which are precisely the continuous $2\pi$-periodic signals $f$ with Fourier coefficients $\hat{f}_0 = \iprod{v_0}{\phi_0}$ and $\hat{f}_k=\iprod{v_0}{\phi_k} \abs{k}^{-3/5}$ for $k\neq 0$. As was shown in~\cite[Sec. 3]{PauPoh12a}, the choice of the space $W_\ga$ of the initial states $v_0\in W_\ga$ determines the smoothness properties of the generated signals $\yref(t)$.

We will now construct a controller that contains an internal model of the exosystem in the sense that $(\mc{G}_1,\mc{G}_2,K)$ satisfy the \Gconds\ in Section~\ref{sec:Gconds}. Since $\dim Y=2$, we must include two copies of $S$ in the controller. We let $Z_1$ be a Banach space, choose $Z_2=W\times W$, and $Z=Z_1\times Z_2$. The operators $\mc{G}_1$, $\mc{G}_2$, and $K$ of the controller are chosen to be of the form
  \eq{
  \mc{G}_1 = \pmat{R_1&R_2\\0& G_1}, \qquad \mc{G}_2 =  \pmat{R_3\\G_2}, 
  \qquad K = \bigl(K_1,  ~K_2\bigr).
  }
   The operators $G_1$ and $G_2$ are defined as
  \eq{
  G_1 = \pmat{S&0\\0&S}, \qquad
  G_2 = \pmat{g_1&0\\0&g_2},
  }
  where $\Dom(G_1) = \Dom(S)\times \Dom(S)$ and $g_1,g_2\in W$ are such that $\iprod{g_1}{\phi_k}\neq 0$ and $\iprod{g_2}{\phi_k}\neq 0$ for all $k\in \Z$. The operator $G_1$ contains the copies of the signal generator.
  The operators $R_1$, $R_2$, $R_3$, and $K$ in the controller can be used in stabilizing the closed-loop system. They should be chosen in such a way that $\mc{G}_1$ with a suitable domain generates a strongly continuous semigroup on $Z$, the closed-loop system operator $A_e$ generates a strongly stable semigroup, and the Sylvester equation $\Sigma S = A_e \Sigma + B_e$ has a solution $\Sigma \in \Lin(W_\ga,X_e)$ for some $\ga\geq 0$. The next lemma shows that the controller satisfies the \Gconds, and thus by Theorem~\ref{thm:GcondIMP} the robust output regulation problem is solved if the closed-loop system stability is achieved.

\begin{lemma} 
  \label{lem:compositeGconds}
  The controller $(\mc{G}_1,\mc{G}_2,K)$ satisfies the \Gconds.
\end{lemma}

\begin{proof}
  Let $y = (y_1,y_2)^T\in \ker(\mc{G}_2)$. Then $0= G_2 y = \bigl( {g_1 y_1 \atop g_2 y_2} \bigr)$. Since $g_1,g_2\neq 0$, we must have $y = (y_1,y_2)^T=0$. This concludes $\ker(\mc{G}_2)=\set{0}$.

  Let $k\in \Z$ and $z = (z_1,z_2)^T \in \ran(i\gw_k - \mc{G}_1)\cap \ran(\mc{G}_2)$. Then there exist $(z_1^1,z_2^1)^T \in \Dom(\mc{G}_1)$ and $y = (y_1,y_2)^T\in Y$ such that
  \eq{
  \pmat{z_1\\z_2} = \left[ i\gw_k -\pmat{R_1& R_2\\0& G_1} \right] \pmat{z_1^1\\z_2^1} = \pmat{R_3\\G_2}y.
  }
  Due to the structure of $\mc{G}_1$, we necessarily have $z_2^1\in \Dom(G_1) = \Dom(S)\times \Dom(S)$, and $z_2 = (i\gw_k - G_1)z_2^1 = G_2 y$. Denote $\psi_k^1 = (\phi_k,0)^T$ and $\psi_k^2 = (0,\phi_k)^T$. Then clearly $G_1 \psi_k^l = i\gw_k \psi_k^l$ for $l=1,2$. Since $G_1$ is skew-adjoint, we have 
  \eq{
  \iprod{z_2}{\psi_k^l} 
  =\iprod{(i\gw_k - G_1)z_2^1}{\psi_k^l} 
  =\iprod{z_2^1}{(-i\gw_k + G_1)\psi_k^l} 
  =\iprod{z_2^1}{(-i\gw_k + i\gw_k)\psi_k^l} =0
  }
for $l=1,2$ and, on the other hand,
  \eq{
  \iprod{z_2}{\psi_k^1} 
  = \Iprod{G_2y}{\left({\phi_k\atop 0}\right)}
  = y_1\iprod{g_1}{\phi_k}, 
  \qquad 
  \iprod{z_2}{\psi_k^2} 
  = \Iprod{G_2y}{\left({0\atop \phi_k}\right)}
  = y_2\iprod{g_2}{\phi_k}.
  }
  Combining these equations we have $y_l\iprod{g_l}{\phi_k}=0$ for $l=1,2$. Since $\iprod{g_1}{\phi_k}\neq 0$ and $\iprod{g_2}{\phi_k}\neq 0$ by assumption, we must have $y=(y_1,y_2)^T=0$. This concludes $z_2=\mc{G}_2 y = 0$, and further shows that~\eqref{eq:Gconds2} is satisfied.

  Since $n_k=1$ for all $k\in\Z$, we have $\ker(i\gw_k - \mc{G}_1)^{n_k-1} = \set{0}$ for all $k\in\Z$, and thus the condition~\eqref{eq:Gconds3} is trivially satisfied.
\end{proof}

\section{Conclusions}
\label{sec:conclusions}

In this paper we have studied the theory of robust output regulation for distributed parameter systems with unbounded input and output operators. In particular, we have extended the internal model principle for the p-copy internal model as well as for the \Gconds\ for this class of infinite-dimensional systems together with infinite-dimensional block diagonal exosystems. Due to the more general setting, it was not possible to repeat the earlier proofs of the internal model principle. Instead, the proofs presented in this paper make use of Theorem~\ref{thm:Robchareqns}, which also provides a way of testing the robustness of a controller with respect to specific perturbations $(\tilde{A},\tilde{B},\tilde{C},\tilde{D},\tilde{E},\tilde{F})\in \Ops$.

The most important topics for future research are the robust controller design and the stabilization of the closed-loop system. The techniques used previously in~\cite{HamPoh00,PauPoh12a} are not applicable without modifications in the case of unbounded control and observation operators in the plant.
In~\cite{PauPoh13b} it was shown that many of the technical assumptions related to the solvability of the Sylvester equations can be simplifed if the controller can achieve polynomial closed-loop stability. Therefore, designing controllers for polynomial stabilization of the closed-loop system is an important research problem.

\subsection*{Acknowledgement}

The authors would like to thank Professor Hans Zwart for helpful advice on regular linear systems, and especially for proposing the approach for verifying Assumption~\ref{ass:CLsysass}.

\appendix

\section{Properties of the Exosystem and the Proof of \\Lemma~\ref{lem:Sylsplit}}
\label{sec:AppA}

\begin{lemma}
  \label{lem:exonondecay}
  Let $\tilde{X}$ be a normed linear space and let $\ga\geq 0$. The infinite-dimensional exosystem has the property that if $Q\in \Lin(W_\ga, \tilde{X})$, then
  \eqn{
  \label{eq:exonondecay}
  QT_S(t) v_0\stackrel{t\rightarrow \infty}{\longrightarrow} 0 \qquad \mbox{for all} \quad v_0\in W_\ga
  }
 if only if $Q=0$.
\end{lemma}

\begin{proof}
  It is clearly sufficient to show that the property~\eqref{eq:exonondecay} implies $Q=0$. To this end, assume~\eqref{eq:exonondecay} is satisfied, and let $k\in\Z$ and $v_0\in P_kW=\Span\set{\phi_k^l}_{l=1}^{n_k}$.
  Now
  \begin{subequations}
  \label{eq:suffcondJordanblockfun}
  \eqn{
  QT_S(t)v_0 &= e^{i\gw_k t}\sum_{l=1}^{n_k} \iprod{v_0}{\phi_k^l} \sum_{j=1}^l \frac{t^{l-j}}{(l-j)!} Q\phi_k^j\\[1ex]
  &=
  e^{i\gw_k t} \sum_{j=0}^{n_k-1} t^j \cdot \frac{1}{j!} \sum_{l=j+1}^{n_k} \iprod{v_0}{\phi_k^l}Q\phi_k^{l-j}
  }
\end{subequations}
  Since $QT_S(t)v_0\rightarrow 0$, it is easy to see that we must have
  \ieq{
  \sum_{l=j+1}^{n_k} \iprod{v_0}{\phi_k^l}Q\phi_k^{l-j}=0 
  }
  for all $j\in \List[0]{n_k-1}$.
  However, by~\eqref{eq:suffcondJordanblockfun} this also implies $QT_S(t_0)v_0=0$ for all $t_0\geq 0$, and in particular $Qv_0 = QT_S(0)v_0=0$. Since $k\in\Z$ and $v_0\in P_kW$ were arbitrary, this shows that $Q\phi_k^l = 0$ for all $k\in\Z$ and $l\in \List{n_k}$. Since $\set{\phi_k^l}_{kl}$ is a basis of $W_\ga$, this concludes $Q=0$.
\end{proof}

\textit{Proof of Lemma~\textup{\ref{lem:Sylsplit}}.}
Assume that $(\tilde{A},\tilde{B},\tilde{C},\tilde{D},\tilde{E},\tilde{F})$ satisfy parts (a) and (b) of Assumption~\ref{ass:pertclass} and let $k\in\Z$. 
  
  We begin by showing that (a) implies (b).
  Let $\Sigma = (\Pi,\Gamma)^T \in \Lin(W_\ga,X_e)$ be such that $\ran(\Sigma P_k)\subset \Dom(\tilde{A}_e)$ and $\Sigma P_k S = \tilde{A}_e \Sigma P_k +\tilde{B}_eP_k$. We have $\ran(\Sigma P_k)\subset \Dom(\tilde{A}_e)\subset \Dom(\tilde{C})\times \Dom(\mc{G}_1)$, which implies
  $\Pi \phi_k^l\in \Dom(\tilde{C})$, $\Gamma\phi_k^l\in \Dom(\mc{G}_1)$ for every $l\in \List{n_k}$.
  For all $l\in \List[2]{n_k}$ we have and (using $S\phi_k^1 = i\gw_k \phi_k^1$ and $S\phi_k^l = i\gw_k \phi_k^l + \phi_k^{l-1}$) 
\begin{subequations}
  \label{eq:SylsplitSylexpand}
  \eqn{
  \pmat{\tilde{E}\phi_k^1\\\mc{G}_2 \tilde{F}\phi_k^1} 
  &= B_e \phi_k^1
  = \Sigma S \phi_k^1 - \tilde{A}_e \Sigma \phi_k^1 
  = (i\gw_k - \tilde{A}_e)\Sigma \phi_k^1 \\
  &=\pmat{ (i\gw_k- \tilde{A}_{-1})\Pi \phi_k^1 - \tilde{B}K\Gamma \phi_k^1 \\ (i\gw_k -  \mc{G}_1)\Gamma \phi_k^1 - \mc{G}_2(\tilde{C}\Pi+\tilde{D}K\Gamma) \phi_k^1}\\
  \pmat{\tilde{E}\phi_k^l \\\mc{G}_2 \tilde{F}\phi_k^l} 
  &= B_e \phi_k^l
  = \Sigma S \phi_k^l - \tilde{A}_e \Sigma \phi_k^l 
  = (i\gw_k - \tilde{A}_e)\Sigma \phi_k^l + \Sigma \phi_k^{l-1}\\
  &=\pmat{ (i\gw_k- \tilde{A}_{-1})\Pi \phi_k^l - \tilde{B}K\Gamma \phi_k^l + \Pi \phi_k^{l-1} \\ (i\gw_k -  \mc{G}_1)\Gamma \phi_k^l - \mc{G}_2(\tilde{C}\Pi+\tilde{D}K\Gamma) \phi_k^l + \Gamma \phi_k^{l-1} }.
  }
\end{subequations}
We have $i\gw_k\in \rho(\tilde{A}) = \rho(\tilde{A}_{-1})$, and we denote $\tilde{R}_k = R(i\gw_k,\tilde{A}_{-1})$ for brevity.
The first lines of the equations~\eqref{eq:SylsplitSylexpand} recursively imply that for $l\in \List[2]{n_k}$ we have
 \eq{
 \Pi \phi_k^1 &= \tilde{R}_k \left( \tilde{B}K\Gamma \phi_k^1 + \tilde{E}\phi_k^1 \right)\\
 \Pi \phi_k^l
 &= \tilde{R}_k \left( \tilde{B}K\Gamma \phi_k^l + \tilde{E}\phi_k^l  - \Pi \phi_k^{l-1}\right) \\
 &= \tilde{R}_k \left( \tilde{B}K\Gamma \phi_k^l + \tilde{E}\phi_k^l \right) - \tilde{R}_k^2 \left( \tilde{B}K\Gamma \phi_k^{l-1} + \tilde{E}\phi_k^{l-1}  -  \Pi \phi_k^{l-2} \right)  \\
 &= \cdots 
 = \sum_{j=0}^{l-1} (-1)^j \tilde{R}_k^{j+1} \left( \tilde{B} K\Gamma \phi_k^{l-j} + \tilde{E}\phi_k^{l-j} \right).
 }
 In vector notation this is precisely~\eqref{eq:Sylsplit2}.  
 Substituting $\Pi\phi_k^l$ into the second lines of the equations~\eqref{eq:SylsplitSylexpand} we see that for all $l\in \List[2]{n_k}$ we have
  \eq{
  \MoveEqLeft[2] (i\gw_k -  \mc{G}_1)\Gamma \phi_k^1 =    \mc{G}_2(\tilde{C}\Pi \phi_k^1 +\tilde{D}K\Gamma\phi_k^1 +  \tilde{F}\phi_k^1 ) \\
  = \,& \mc{G}_2\left[  (\tilde{C}\tilde{R}_k\tilde{B}+\tilde{D}) K\Gamma \phi_k^1 +  \tilde{C} \tilde{R}_k\tilde{E}\phi_k^1 + \tilde{F}\phi_k^1 \right] \\
  = \,&\mc{G}_2\left(  \tilde{P}(i\gw_k) K\Gamma \phi_k^1 +  \tilde{C} \tilde{R}_k\tilde{E}\phi_k^1 + \tilde{F}\phi_k^1 \right) \\
  \MoveEqLeft[2] (i\gw_k -  \mc{G}_1)\Gamma \phi_k^l + \Gamma \phi_k^{l-1} =     \mc{G}_2(\tilde{C}\Pi \phi_k^l +\tilde{D}K\Gamma\phi_k^l +  \tilde{F}\phi_k^l ) \\
  = \,& \mc{G}_2\left[  \sum_{j=0}^{l-1} (-1)^j \tilde{C}\tilde{R}_k^{j+1} \left( \tilde{B} K\Gamma \phi_k^{l-j} + \tilde{E}\phi_k^{l-j} \right) +\tilde{D}K\Gamma\phi_k^l +  \tilde{F}\phi_k^l\right]  \\
  = \,& \mc{G}_2\left[  \tilde{P}(i\gw_k) K\Gamma \phi_k^l + \sum_{j=1}^{l-1} (-1)^j \tilde{C}\tilde{R}_k^{j+1}  \tilde{B} K\Gamma \phi_k^{l-j} 
    + \sum_{j=0}^{l-1} (-1)^j \tilde{C}\tilde{R}_k^{j+1} \tilde{E}\phi_k^{l-j}  +  \tilde{F}\phi_k^l\right] . 
  }
  In vector notation this is exactly~\eqref{eq:Sylsplit1}.
  This concludes that (b) is satisfied.

  We will now show that (b) implies (a).
  To this end, assume $\Sigma = (\Pi,\Gamma)^T \in \Lin(W_\ga,X_e)$ is such that $\ran(\Sigma P_k) \subset \Dom(\tilde{C})\times \Dom(\mc{G}_1)$ and~\eqref{eq:Sylsplit} are satisfied.
For all $l\in \List{n_k}$ we have $\Pi\phi_k^l \in \Dom(\tilde{C})$ and $\Gamma \phi_k^l\in \Dom(\mc{G}_1)$, and as above we can see that the equations~\eqref{eq:Sylsplit} imply
  \eq{
  \MoveEqLeft \Sigma S  \phi_k^1 - \tilde{A}_e \Sigma \phi_k^1 
  = i\gw_k\Sigma  \phi_k^1 - \tilde{A}_e \Sigma \phi_k^1 \\
  &=\pmat{ (i\gw_k- \tilde{A}_{-1})\Pi \phi_k^1 - \tilde{B}K\Gamma \phi_k^1\\ (i\gw_k -  \mc{G}_1)\Gamma \phi_k^1 - \mc{G}_2(\tilde{C}\Pi+\tilde{D}K\Gamma) \phi_k^1}
  = \pmat{\tilde{E}\phi_k^1\\\mc{G}_2 \tilde{F}\phi_k^1} 
  = \tilde{B}_e \phi_k^1\\
   \MoveEqLeft \Sigma S \phi_k^l - \tilde{A}_e \Sigma \phi_k^l 
   = (i\gw_k   - \tilde{A}_e )\Sigma \phi_k^l + \Sigma \phi_k^{l-1}\\
   &   =\pmat{ (i\gw_k- \tilde{A}_{-1})\Pi \phi_k^l - \tilde{B}K\Gamma \phi_k^l + \Pi \phi_k^{l-1}\\ (i\gw_k -  \mc{G}_1)\Gamma \phi_k^l - \mc{G}_2(\tilde{C}\Pi+\tilde{D}K\Gamma) \phi_k^l + \Gamma \phi_k^{l-1}}
   = \pmat{\tilde{E}\phi_k^l\\\mc{G}_2 \tilde{F}\phi_k^l} 
   = \tilde{B}_e \phi_k^l.
 }
Since $\tilde{B}_e\phi_k^l\in X_e$ and $\Sigma \phi_k^l\in X_e$, the second and fourth line above also show that $(\Pi \phi_k^l,\Gamma \phi_k^l)^T\in \Dom(i\gw_k - \tilde{A}_e) =\Dom(\tilde{A}_e)$ for all $l\in \List{n_k}$, and thus $\ran(\Sigma P_k)\subset  \Dom(\tilde{A}_e)$.
 This concludes that $\Sigma P_k$ is a solution of the Sylvester equation $\Sigma_k S = \tilde{A}_e \Sigma_k + \tilde{B}_eP_k$, and thus (a) is satisfied.

 If $\Sigma = (\Pi,\Gamma)^T$ satisfies~\eqref{eq:Sylsplit}, then~\eqref{eq:Sylsplit2} implies 
  \eq{
  \MoveEqLeft[2] \tilde{C_e} \Sigma \phi_k^1 + \hspace{-.2ex} \tilde{D}_e \phi_k^1
  =  \tilde{C}\Pi \phi_k^1 \hspace{-.2ex} + \hspace{-.2ex} \tilde{D} K \Gamma \phi_k^1 \hspace{-.2ex} + \hspace{-.2ex} \tilde{F}\phi_k^1
  = \tilde{C}\tilde{R}_k\bigl( \tilde{B} K\Gamma \phi_k^1 + \tilde{E}\phi_k^1  \bigr) \hspace{-.2ex}+ \hspace{-.2ex}\tilde{D} K \Gamma \phi_k^1 \hspace{-.2ex} + \hspace{-.2ex} \tilde{F}\phi_k^1\\
  &=  \bigl( \tilde{C}\tilde{R}_k \tilde{B}  + \tilde{D}  \bigr)K \Gamma \phi_k^1 +    \tilde{C} \tilde{R}_k \tilde{E}\phi_k^1 +\tilde{F}\phi_k^1 
  =  \tilde{P}(i\gw_k)K \Gamma\phi_k^1 +  \tilde{C} \tilde{R}_k \tilde{E}\phi_k^1 +\tilde{F}\phi_k^1 \\
\MoveEqLeft[2]  \tilde{C_e} \Sigma \phi_k^l + \tilde{D}_e \phi_k^l
  =  \tilde{C}\Pi \phi_k^l + \tilde{D} K \Gamma \phi_k^l + \tilde{F}\phi_k^l\\
  &=  \sum_{j=0}^{l-1} (-1)^j \tilde{C}\tilde{R}_k^{j+1} \left( \tilde{B} K\Gamma \phi_k^{l-j} + \tilde{E}\phi_k^{l-j} \right) + \tilde{D} K \Gamma \phi_k^l + \tilde{F}\phi_k^l\\
&=  \tilde{P}(i\gw_k) K\Gamma \phi_k^l + \sum_{j=1}^{l-1} (-1)^j \tilde{C}\tilde{R}_k^{j+1}  \tilde{B} K\Gamma \phi_k^{l-j} 
+ \sum_{j=0}^{l-1} (-1)^j \tilde{C}\tilde{R}_k^{j+1} \tilde{E}\phi_k^{l-j}  + \tilde{F}\phi_k^l.
  }
  In vector notation, this is exactly~\eqref{eq:Sylsplit3}.  

It remains to show that (c) and (d) are equivalent. Assume that $(\tilde{A},\tilde{B},\tilde{C},\tilde{D},\tilde{E},\tilde{F})\in\Ops$. If (c) is satisfied, we have from Theorem~\ref{thm:Sylprop} that for every $k\in\Z$ the operator $\Sigma P_k$ is a solution of the Sylvester equation $\Sigma_k S = \tilde{A}_e \Sigma_k + \tilde{B}_e P_k$. Therefore, part (d) follows directly from the fact that (a) implies (b).

Assume now that (d) is satisfied, i.e., the operator $\Sigma: \Dom(\Sigma)\subset W \rightarrow X_e$ is such that $\ran(P_k) \subset \Dom(\Sigma)$, $\ran(\Sigma P_k)\subset \Dom(\tilde{C})\times \Dom(\mc{G}_1)$,  and~\eqref{eq:Sylsplit} are satisfied for all $k\in\Z$. 
Since $(\tilde{A},\tilde{B},\tilde{C},\tilde{D},\tilde{E},\tilde{F})\in\Ops$, we have from Assumption~\ref{ass:pertclass} that there exists $\tilde{\Sigma}\in \Lin(W_\ga,X_e)$ such that $\tilde{\Sigma}(W_{\ga+1})\subset \Dom(\tilde{A}_e)$ and  $\tilde{\Sigma} S  = \tilde{A}_e \tilde{\Sigma} + \tilde{B}_e$. 
  We will show that $\Sigma = \tilde{\Sigma}$.

  For $k\in\Z$ the equivalence of (a) and (b) implies that $\Sigma P_k$ is a solution of the Sylvester equation $\Sigma_k S = \tilde{A}_e \Sigma_k + \tilde{B}_e P_k$.
  However, by Theorem~\ref{thm:Sylprop} these equations have unique solutions $\tilde{\Sigma}P_k$, and thus we must have $\Sigma P_k = \tilde{\Sigma }P_k$ for all $k\in\Z$.
  This in particular implies that $\Sigma v = \tilde{\Sigma} v$ for all $v$ in the space
  \eq{
  W_\infty = \Bigl\{ \sum_{\abs{k}\leq N} \sum_{l=1}^{n_k} v_{kl} \phi_k^l \; \Bigm| \; N\in \N, ~ v_{kl}\in\C \Bigr\}.
  }
  This space satisfies $W_\infty\subset W_\ga$, and for all
  $v\in W_\infty$ the property $\Sigma v = \tilde{\Sigma} v$ implies
  \eq{
  \norm{\Sigma v} = \norm{\tilde{\Sigma} v} \leq \norm{\tilde{\Sigma}}_{\Lin(W_{\ga},X_e)} \norm{v}_\ga.
  }
  The space $W_\infty$ is dense in $W_\ga$. Therefore
$\Sigma$ has a unique extension in $\Lin(W_\ga,X_e)$, and this extension is equal to $\tilde{\Sigma}$.
This finally concludes that $\Sigma$ (or its extension) satisfies $\Sigma (W_{\ga+1})\subset \Dom(\tilde{A}_e)$ and it is a solution of the Sylvester equation $\Sigma S = \tilde{A}_e \Sigma + \tilde{B}_e $.
 $\Box$

\end{document}